\documentclass[12pt]{amsart}
\usepackage{hyperref}
\usepackage{amssymb}
\usepackage{amsmath}
\numberwithin{equation}{section}
\usepackage{color}
\usepackage{graphicx}
\usepackage{epstopdf}
\usepackage{todonotes} 

\newtheorem{thm}{Theorem}[section]
\newtheorem{lemma}[thm]{Lemma}
\newtheorem{cor}[thm]{Corollary}

\theoremstyle{definition}

\theoremstyle{remark}
\newtheorem{remark}[thm]{Remark}

\theoremstyle{rem}

\newcommand\Q{{\mathbb Q}}
\newcommand\Z{{\mathbb{Z}}}

\newcommand\PP{{\mathbb{P}}}
\newcommand\N{{\mathbb{N}}}

\newcommand\bq{\begin{equation}}
\newcommand\eq{\end{equation}}
\newcommand\beq{\begin{eqnarray*}}
\newcommand\eeq{\end{eqnarray*}}
\newcommand\ben{\begin{enumerate}}
\newcommand\een{\end{enumerate}}
\newcommand\bit{\begin{itemize}}
\newcommand\eit{\end{itemize}}

\newcommand\des{{\rm des}}
\newcommand\cdes{{\rm cdes}}
\newcommand\asc{{\rm asc}}

\newcommand\exc{{\rm exc}}

\newcommand\maj{{\rm maj}}

\newcommand\sg{{\mathfrak S}}
\newcommand\Des{{\rm DES}}
\newcommand\Asc{{\rm ASC}}

\newcommand\x{{\mathbf x}}

\newcommand\wt{{\rm wt}}

\newcommand\sps{{\bf \rm  ps}}

\begin{document}

\title[Smirnov word enumerators]
{On enumerators of Smirnov words by descents and cyclic descents}
\author[Ellzey]{Brittney Ellzey}
\address{Department of Mathematics, University of Miami, Coral Gables, FL 33124}
\email{ellzey@math.miami.edu}

\author[Wachs]{Michelle L. Wachs$^2$}
\address{Department of Mathematics, University of Miami, Coral Gables, FL 33124}
\email{wachs@math.miami.edu}
\thanks{$^{2}$Supported in part by NSF Grant
DMS  1502606}

\subjclass[2010]{05E05, 05A05, 05A15, 05A30}

\date{Nov. 29, 2018; revised April 26, 2019}

\dedicatory{Dedicated to the memory of Jeff Remmel}

\begin{abstract} A Smirnov word is a word over the positive integers in which adjacent letters must be different. 
A symmetric function enumerating  these words by descent number  arose in  the work of Shareshian and the second named author on $q$-Eulerian polynomials,  where a $t$-analog of a formula of Carlitz, Scoville, and Vaughan for enumerating  Smirnov words is proved.  A symmetric function enumerating a circular version of these words by cyclic descent number  arose in  the work of the first named author on chromatic quasisymmetric functions of directed graphs, where a $t$-analog of a formula of Stanley for enumerating circular Smirnov words is proved.  

In this paper we  obtain new $t$-analogs of the  Carlitz-Scoville-Vaughan formula and the Stanley formula in which the roles of   descent number and cyclic descent number are switched.   These formulas show that the  Smirnov word enumerators are  polynomials in $t$ whose coefficients are e-positive symmetric functions.  We also obtain expansions   in the power sum basis and the fundamental quasisymmetric function basis, complementing earlier results of Shareshian and the authors.     

Our work relies on studying refinements of the Smirnov word enumerators that count certain restricted classes of Smirnov words by descent number.   Applications to  variations of $q$-Eulerian polynomials and  to the chromatic quasisymmetric functions introduced by Shareshian and the second named author  are also presented.
\end{abstract}
\maketitle

\vbox{
\tableofcontents
}
\section{Introduction}
We consider words $w=w_1w_2 \dots w_n$ over the alphabet of positive integers $\PP$ with no adjacent repeated letters; that is $w_i \ne w_{i+1}$ for all $i \in [n-1] :=\{1,\dots,n-1\}$.  We refer to these words as  {\em Smirnov words} as is often done in the literature; see e.g.  \cite{GJ, FS, lsw, a1,  lr2, FHP, m}.

For $n \ge 1$, let $W_n$  be the set of Smirnov words of length $n$.
Now define the Smirnov word enumerator $$W_n(\x):= \sum_{w \in W_n} x_w,$$ where
$\x:=x_1,x_2,\dots$ is a sequence of indeterminates and $x_w := x_{w_1}x_{w_2}\cdots x_{w_n}$.  Clearly $W_n(\x)$ is a symmetric function.
Carlitz, Scoville, and Vaughan \cite[equation (7.12)]{csv}  derived the generating function formula 
\begin{equation}\label{goodcarleq} \sum_{n \ge 1} W_n({\bf x}) z^n =\frac{ \sum_{i \ge 1}i \,\, e_i(\x) z^i } {1 - \sum_{i \ge 2} (i-1) e_i(\x) z^i}\,,\end{equation} 
where $e_i(\x)$ is the elementary symmetric function of degree $i$.  An important consequence of this formula is that 
 $W_n(\x)$ is $e$-positive, which means that when expanded in the elementary symmetric function basis for the ring of symmetric functions, the coefficients are nonnegative.

The symmetric function $W_n(\x)$ was also considered by Stanley \cite{st4} in the context of chromatic symmetric functions and by Dollhopf, Goulden, and Greene  \cite{dgg} in the context of pair avoiding word enumerators.  Stanley also considered  a circular version of Smirnov words, that is Smirnov words whose first and last letter are different.  Let $$W^{\ne}_n(\x):= \sum_{\scriptsize \begin{array}{c} w \in W_n \\ w_1 \ne w_n \end{array}} x_w .$$ Stanley  \cite[Proposition 5.4]{st4} 
proved
\begin{equation}\label{circeq} \sum_{n \ge 1} W^{\ne}_n(\x) z^n =  \frac{\sum_{i \ge 1} i(i-1) \,\,e_i(\x) z^i } {1 - \sum_{i \ge 2} (i-1) e_i(\x) z^i}.\end{equation}
It follows from this formula that $W^{\ne}_n(\x)$ is  $e$-positive.

Given any word $w \in \PP^n$, where $n \ge 1$, the {\em descent number} of $w$ is defined by $$\des(w) := |\{i \in [n-1] : w_i > w_{i+1} \}|$$ and the {\em cyclic descent number} is defined by
\bq \label{cirdesdef} \cdes(w) := |\{ i \in [n] : w_i > w_{i+1}\}|,\eq where $w_{n+1} := w_1$.  Now define the refined Smirnov word  enumerators
\begin{eqnarray*} W_n(\x,t) &:=& \sum_{w \in W_n} t ^{\des(w)} x_w \\
\tilde W_n(\x,t) &:=& \sum_{w \in W_n} t ^{\cdes(w)} x_w \\
W^{\ne}_n(\x,t) &:=& \! \! \! \! \! \sum_{\scriptsize \begin{array}{c} w \in W_n \\ w_1 \ne w_n \end{array}} \! \! \! \! t^{\des(w)} x_w \\
\tilde W^{\ne}_n(\x,t) &:=& \! \! \! \! \! \sum_{\scriptsize \begin{array}{c} w \in W_n \\ w_1 \ne w_n \end{array}} \! \! \! \! t^{\cdes(w)} x_w.
\end{eqnarray*} 
The first and fourth of these Smirnov word enumerators have been studied before.  The main objective of this paper is to study  the other two Smirnov word enumerators. We start with a brief review of what is known for $W_n(\x,t)$ and $\tilde W_n^{\ne}(\x,t)$.

\subsection{Summary of known results} The refined Smirnov word enumerator 
$W_n(\x,t) $
 arose in the work  of Shareshian and the second named author on $q$-Eulerian polynomials \cite{sw1}.  
  Stanley (personal communication) observed 
 that the $r=1$ case of \cite[Theorem~1.2]{sw1}  is equivalent to  the following $t$-analog of (\ref{goodcarleq}),
 \bq \label{goodtcarleq} \sum_{n \ge 1} W_n({\bf x},t) z^n = \frac{ \sum_{i \ge 1} [i]_t \,\, e_i(\x) z^i } {1 - \sum_{i \ge 2} t[i-1]_t \,\,e_i(\x) z^i} ,
\eq
 where  $$[n]_t := 1+t + \dots + t^{n-1} = \frac{t^n-1}{t-1}.$$  Indeed, this follows from
 \cite[Equation (7.7) and Theorem 3.6]{sw1}; see  \cite[Section~4.1]{sw6}.
  (For another proof of (\ref{goodtcarleq}), see Remark~\ref{newproof}.)  It follows from (\ref{goodtcarleq}) that $W_n(\x,t)$ is $e$-positive, that is, $W_n(\x,t)$ is a polynomial in $t$ whose  coefficients are $e$-positive symmetric functions.

 Equation (\ref{goodtcarleq}) can be restated as
 \bq \label{goodtcarleq2} 1+ \sum_{n \ge 1} W_n({\bf x},t) z^n = \frac{ (1-t)E(z) }{E(tz) - tE(z)}, \eq
 where 
 $$E(z) := \sum_{n\ge 0} e_n(\x) z^n .$$
 The expression on the right hand side of (\ref{goodtcarleq2}) (or its image under the involution $\omega$ that takes $e_n(\x)$ to the complete homogeneous symmetric function $h_n(\x)$) has arisen in various other contexts.  
Stanley obtained the expression when considering the representation of the symmetric group on the cohomology of the toric variety associated with the dual permutohedron \cite{st2}.  A conjecture of  Shareshian and the second named author \cite{sw5, sw4, sw6}, proved by Brosnan and Chow \cite{BC} and subsequently by Guay-Paquet\cite{gp2},   generalizes the connection between the Smirnov word  enumerator  and the toric varieties to a connection between chromatic quasisymmetric functions and Hessenberg varieties.   The expression also arose in the work of Shareshian and the second named author on  the representation of the  symmetric group on homology of Rees products of posets \cite{sw2}.
 
 By combining (\ref{goodtcarleq2}) with a result of Stembridge \cite{stem1}, one obtains an expansion of $\omega W_n(\x,t) $ in the power sum  
 symmetric functions $p_\lambda(\x)$.  The expansion is given by
\begin{equation} \label{powereq} \omega W_n(\x,t) = \sum_{\lambda \vdash n} \left (A_{l(\lambda)}(t) \prod_{i=1}^{l(\lambda)} [\lambda_i]_t \right )\frac{p_\lambda(\x)}{z_\lambda} ,\end{equation} 
where $A_m(t)$ is the $m$th Eulerian polynomial, $z_\lambda$ is a constant associated with the partition $\lambda=(\lambda_1 \ge \lambda _2 \ge \dots \ge \lambda_{l(\lambda)})$, and $l(\lambda)$ is the length of $\lambda$. 

Recall that the Eulerian polynomials $A_n(t)$ have two well-known combinatorial interpretations, which are given by
\bq \label{combeuldef} A_n(t) = \sum_{\sigma \in \sg_n} t^{\des(\sigma)} = \sum_{\sigma \in \sg_n} t^{\exc(\sigma)},\eq
where $\sg_n$ is the symmetric group on $[n]$, and $\des$ and $\exc$ are MacMahon's classical equidistributed permutation statistics, descent number and excedance number, respectively. 

In \cite{sw,sw1} Shareshian and the second named author introduce a $q$-analog of the $\exc$ interpretation of $A_n(t)$  and in \cite{sw4} they introduce a $q$-analog of the $\des$ interpretation of $A_n(t)$.
 These $q$-analogs are shown to be equal in \cite[Theorem 9.7]{sw4}.  They  are defined by
$$A_n(q,t):= \sum_{\sigma \in \sg_n} q^{\maj(\sigma) - \exc(\sigma)} t^{\exc(\sigma)} = 
\sum_{\sigma \in \sg_n} q^{\maj_{\ge 2}(\sigma^{-1})} t^{\des(\sigma)},$$
 where $\maj$ is MacMahon's classical major index and $\maj_{\ge 2}$ is a permutation statistic whose definition is given in Section~\ref{qeulersec}.  By taking the stable principal specialization of  both sides of (\ref{goodtcarleq2}), the following $q$-analog of  Euler's classical formula is established in  \cite{sw1} for the $\exc$ interpretation of $A_n(q,t)$  and in \cite{sw4} for the $\des$ interpretation:
  \bq \label{qeuleq} 1+ \sum_{n\ge 1} A_n(q,t) \frac{z^n}{[n]_q!} = 
 \frac{(1-t) \exp_q(z)} {\exp_q(tz) - t \exp_q(z)} ,\eq
 where $$[n]_q ! := [n]_q [n-1]_q \dots [1]_q \quad  \mbox{ and } \quad \exp_q(z):= \sum_{n\ge 0} \frac{z^n}{[n]_q!}.$$
 In \cite{ssw}, Sagan, Shareshian and the second named author use the expansion (\ref{powereq}) of $\omega W_n(\x,t)$ in the power sum basis  to show that $A_n(q,t)$ evaluated at  any $n$th root of unity is a polynomial in $t$ with positive integer coefficients.
  For results on cycle-type refinements of the $\exc$ interpretation of $A_n(q,t)$ see \cite{sw4,HW,ssw}.

 The Smirnov word enumerator 
$\tilde W^{\ne}_n(\x,t) $
 arose in the work  \cite{e2,e1,e3}  of the first named author on chromatic quasisymmetric functions of directed graphs. 
The first named author  proves the  $t$-analog of (\ref{circeq}),
\begin{equation}\label{tcirceq} \sum_{n \ge 1} \tilde W^{\ne}_n(\x,t) z^n =  \frac{\sum_{i \ge 2} it[i-1] _t \,\,e_i(\x) z^i } {1 - \sum_{i \ge 2} t[i-1]_t e_i(\x) z^i} \end{equation}
from which $e$-positivity of $ \tilde W^{\ne}_n(\x,t) $ follows.  (A subsequent alternative proof of (\ref{tcirceq}) was given in \cite{AP}.)  As a consequence of a  general result  obtained in \cite{e2,e1,e3} on power sum expansions of chromatic quasisymmetric functions, the first named author also obtains the following expansion  analogous to (\ref{powereq}):

\bq \label{powercycneeq} \omega \tilde W^{\ne}_n(\x,t) = \sum_{\scriptsize \begin{array}{c} \lambda \vdash n \\ l(\lambda) >1\end{array}} \left (nt A_{l(\lambda)-1}(t) \prod_{i=1}^{l(\lambda)} [\lambda_i]_t \right )\frac{p_\lambda}{z_\lambda} \,\, + \,\, nt[n-1]_t \frac{ p_n}{n}.\eq

\subsection{New results} In this paper we obtain results for $\tilde W_n(\x,t)$ and $W^{\ne}_n(\x,t) $, analogous to those described above.  For instance, we prove the $t$-analog of (\ref{goodcarleq}),
\bq  \label{circfulleq}\sum_{n\ge 1} \tilde W_n(\x,t) z^n = \frac{\sum_{i\ge 1}  i t^{i-1} e_i(\x) z^i} {{1 - \sum_{i \ge 2} t[i-1]_t \,\,e_i(\x) z^i}}\, , \eq
and the $t$-analog of (\ref{circeq}),
\bq  \label{noteqeq}\sum_{n\ge 1}  W^{\ne}_n(\x,t) z^n = \frac{\sum_{i\ge 2} ([i]_t +it[i-2]_t) \, e_i(\x) z^i} {{1 - \sum_{i \ge 2} t[i-1]_t \,\,e_i(\x) z^i}} .\eq
From this it follows that $\tilde W_n(\x,t)$ and $W_n^{\ne}(\x,t)$ are symmetric in $\x$ and $e$-positive.
We also obtain  expansions  in the power sum symmetric functions analogous to (\ref{powereq}) and (\ref{powercycneeq}).

Equation (\ref{circfulleq}) can be restated as
 \bq \label{circfulleq2} \sum_{n \ge 1} \tilde W_n({\bf x},t) z^n = \frac{( 1-t)\frac{\partial} {\partial t} E(tz)  }{E(tz) - tE(z)}. \eq
 By specializing   (\ref{circfulleq2}), using an expansion of $\tilde W_n(\x)$ in the fundamental quasisymmetric functions, we obtain the cyclic analog of (\ref{qeuleq}),
 \bq \label {expcirceq} \sum_{n\ge 1} \left (\sum_{\sigma \in \sg_n }  q^{\maj_{\ge 2}(\sigma^{-1})} t^{\cdes(\sigma)} \right) \frac{z^n}{[n]_q!}= \frac{(1-t)\frac{\partial }{\partial t}\exp_q(tz)}{\exp_q(tz) -t\exp_q(z)} .\eq

 Our work relies on studying restricted Smirnov word enumerators that are components of all the Smirnov word enumerators discussed above.  For $n \ge 1$, let
\begin{eqnarray*} W_n^<(\x,t) &:=& \sum_{\scriptsize \begin{array}{c} w \in W_n \\ w_1 < w_n  \end{array}}  t^{\des(w)} x_w\\ 
W_n^>(\x,t) &:=& \sum_{\scriptsize \begin{array}{c} w \in W_n \\ w_1 > w_n  \end{array}} t^{\des(w)} x_w \\
W_n^=(\x,t) &:=& \sum_{\scriptsize \begin{array}{c} w \in W_n \\ w_1 = w_n  \end{array}}  t^{\des(w)} x_w.
\end{eqnarray*}
Clearly 
\begin{eqnarray} \label{refineeq1} W_n(\x,t) &=& W^<_n(\x,t) + W^>_n(\x,t) + W^=_n(\x,t)
\\ \label{refineeq2} \tilde W_n(\x,t) &=& t W^<_n(\x,t) + W^>_n(\x,t) + W^=_n(\x,t) 
\\ \label{refineeq3} W^{\ne}_n(\x,t) &=& W^<_n(\x,t) + W^>_n(\x,t)  \\
 \label{refineeq4}\tilde W^{\ne}_n(\x,t) &=& t W^<_n(\x,t) + W^>_n(\x,t) .\end{eqnarray}

It is an exercise in \cite[Exercise 2.9.11]{GR} that $W^<_n(\x,t)$, $W^>_n(\x,t)$, and  $ W^=_n(\x,t)$ are symmetric in $\x$.  Here we derive results for $W^<_n(\x,t)$, $W^>_n(\x,t)$, and  $ W^=_n(\x,t)$ analogous to those of  $W_n(\x,t)$, which not only establish symmetry, but also $e$-positivity of $W^<_n(\x,t)$ and $W^{>}(\x,t)$.  When appropriately combined they yield the above mentioned results for $\tilde W_n(\x,t)$ and $ W_n^{\ne}(\x,t)$.
They also enable us to recover the previous results for  $W_n(\x,t)$ and $ \tilde W_n^{\ne}(\x,t)$.   However they  do not provide new proofs of  the previous results   since their proofs  rely on these results.

The Smirnov word enumerator is an example of a chromatic quasisymmetric function.  (The chromatic quasisymmetric functions are  a refinement  of Stanley's chromatic symmetric functions, which were introduced by Shareshian and the second named author in \cite{sw5,sw4}). Indeed, $W_n(\x,t)$ is the chromatic quasisymmetric function of the naturally labeled path graph with $n$ nodes.  In this paper our results for $W^<_n(\x,t)$ and $W^>_n(\x,t)$ are used to obtain new results in the study of chromatic quasisymmetric functions.  For instance,  we use $e$-positivity of $W^<_n(\x,t)$ and $W^>_n(\x,t)$ to establish $e$-positivity of the chromatic quasisymmetric function of the labeled cycle $C_n$, providing an example of an $e$-positive chromatic quasisymmetric function not covered by the refinement of the Stanley-Stembridge $e$-positivity conjecture appearing in \cite{sw5,sw4} nor by the directed graph version appearing in \cite{e2,e1,e3}.

Smirnov words have been used in the literature  to enumerate unconstrained words; see e.g. \cite{FS,lr1,lr2,m}. 
In a forthcoming paper we use  results discussed in this paper to obtain analogous results for unconstrained words.

\section{Expansion in the elementary symmetric functions} \label{esec}

In this section we derive  formulas that refine (\ref{goodtcarleq}) and (\ref{tcirceq}) and then use the refinements to prove   (\ref{circfulleq}) and (\ref{noteqeq}).

\begin{thm}\label{mainth}  Let
 \begin{eqnarray} \label{Ddef}
 D(\x,t,z) &:=& 1 - \sum_{i \ge 2} t[i-1]_t e_i(\x) z^i , \\ 
\nonumber a_i(t) &:=& {d \over d t} [i]_t \hspace{.5in} =  \sum_{j=0}^{i-2} (j+1) t^j , \\
\nonumber b_i(t) &:=& t^{i -1}a_i(t^{-1}) \hspace{.1in}= \sum_{j=1}^{i-1} (i-j) t^j, \\
\nonumber c_i(t) &:=&  it[i-2]_t ,
 \end{eqnarray}
 for all $i\ge 2$.  Then
\begin{eqnarray} \label{lesseq} \qquad \sum_{n \ge 1} W^<_n({\bf x},t) z^n &=&  {1 \over D(\x,t,z)} \sum_{i \ge 2} a_i(t) \,\,e_i(\x) z^i \\ 
 \label{greateq}\sum_{n \ge 1} W^>_n({\bf x},t) z^n &=&  {1 \over D(\x,t,z)} \sum_{i \ge 2}b_i(t) \,\, e_i(\x) z^i 
 \\
 \label{equaleq} \sum_{n \ge 1} W^=_n({\bf x},t) z^n &= & {1 \over D(\x,t,z)} (e_1(\x) z -  \sum_{i \ge 2}  c_i(t)\,\, e_i(\x) z^i).
 \end{eqnarray}
 \end{thm}

 Before proving the theorem, we observe that
\begin{eqnarray*} a_i(t)+b_i (t)&=& 1+  (i+1) t + (i+1) t^2 + \dots + (i+1) t^{i-2} + t^{i-1}
\\ &=& [i]_t + i t[i-2]_t  .\end{eqnarray*}
Hence,  
$$a_i(t)+b_i(t)- c_i(t) = [i]_t,$$ which shows that Theorem~\ref{mainth} refines  (\ref{goodtcarleq}) since $W_n^<(\x,t) + W_n^>(\x,t) + W_n^=(\x,t)=W_n(\x,t)$.
Also $$ta_i(t)+b_i(t) = it[i-1]_t,$$ which shows that Theorem~\ref{mainth} also refines (\ref{tcirceq}) since $t W_n^<(\x,t)  + W_n^>(\x,t)  = \tilde W^{\ne}_n(\x,t)$.

 We prove (\ref{equaleq}) first.  Then we use  (\ref{equaleq}), (\ref{goodtcarleq}), and (\ref{tcirceq})   to derive  (\ref{lesseq}).  
 Equation (\ref{greateq}) follows from (\ref{lesseq}).  Our proof of 
 (\ref{equaleq}) uses the transfer-matrix method discussed in~\cite[Section 4.7]{st5} and borrows ingredients from the first named author's proof of (\ref{tcirceq}) in \cite{e1}.
 
 Before presenting the proof of (\ref{equaleq}), we give a brief review of the transfer matrix method. 
A {\em walk} of length $n$ on a directed graph $G=([k],E)$ is a sequence of vertices $v_0, v_1,\dots, v_n$ 
such that $(v_{i-1},v_{i}) \in E$ for all $i \in [n]$.  A walk is {\em closed} if $v_0 = v_n$.
We attach weights  in some commutative ring $R$ to the edges of $G$.  Let $\wt: E \to R$ be the 
weight function. Now  define the weight $\wt(w)$ of a walk $w:= v_0, v_1,\dots, v_n$ to be the 
product $\wt(v_0,v_1)\wt(v_1,v_2) \dots \wt(v_{n-1},v_n)$.  For each $i,j \in [k]$, define $\mathcal W_{i,j,n}$ to be the set of walks of length $n$ from $i$ 
to $j$ and let
$$U_{i,j,n} := \sum_{w \in \mathcal W_{i,j,n}} \wt(w).$$  

The  transfer-matrix method enables one to express the generating function $\sum_{n\ge 0} U_{i,j,n} z^n$ in terms of the adjacency matrix $A$ for the edge weighted digraph $G$. That is, $A$ is the $k\times k$ matrix whose $(i,j)$-entry is  
$$A_{i,j} := \begin{cases} \wt(i,j)  &\mbox{if }(i,j) \in E \\ 0 &\mbox{otherwise.} \end{cases}$$
Theorem~4.7.2 of \cite{st5} states that for all $i,j \in [k]$,

\begin{equation} \label{transmeq} \sum_{n\ge 0} U_{i,j,n} z^n = \frac{(-1)^{i+j} \det(I-zA:j,i)}{\det(I-zA)},\end{equation} where $(B:i,j)$ is the matrix obtained from $B$ by removing row $i$ and column $j$. 

\begin{proof}[Proof of (\ref{equaleq})] As in \cite{st4} and \cite{e1}, we view a Smirnov word $w_1w_2\dots w_n$ over the alphabet $[k]$  as a walk $w_1, w_2,\dots ,w_n$ of length $n-1$ on the digraph $G=([k], E)$, where $$E=\{(i,j) : i,j \in [k]  \mbox{ and } i \ne j\}.$$  We let the edge weights belong to the commutative ring  $\Z[x_1,\dots,x_k,t]$ and for all $(i,j) \in E$, and we set
$$\wt(i,j) := \begin{cases}  x_j &\mbox{if } i<j \\
     tx_j &\mbox{if } i>j . \end{cases}$$
     
Note that if $w$ is a Smirnov word  over the alphabet $[k]$ then 
$$t^{\des(w)} x_w = x_{w_1} \wt(w),$$ where $w_1$ is the first letter of $w$.  
Hence $$W^=_n(x_1,\dots,x_k,t):=W^=_n(x_1,\dots,x_k,0,0,\dots,t)= \sum_{i=1}^k x_i U_{i,i,n-1}.$$
It follows from this and (\ref{transmeq}) that
\begin{eqnarray} \nonumber \sum_{n\ge 1} W^=_n(x_1,\dots,x_k,t) z^n &=& z \sum_{i=1}^k x_i \sum_{n\ge 0} U_{i,i,n} z^n \\ \nonumber
&=&  z \sum_{i=1}^k x_i \frac{ \det(I-zA:i,i)}{\det(I-zA)}
\\ \label{deteq} &=& \frac{z \sum_{i=1}^k x_i \det(I-zA:i,i)}{\det(I-zA)},
\end{eqnarray}
where  
$$A= \bmatrix  0 & x_2 &x_3 & \dots & x_k \\ t x_1 & 0 & x_3&  \dots & x_k \\ tx_1 & tx_2 & 0 & \dots & x_k \\   \vdots & \vdots & \vdots & \ddots & \vdots 
\\ tx_1 & tx_2 & tx_3  & \dots &0\endbmatrix .$$

In \cite[Proof of Theorem~6.1]{e1} the first named author proves that\footnote{This is obtained from the formula in \cite{e1} by  replacing $t$ with $t^{-1}$ and each $x_i$ with $t x_i$}  
\bq \label{elzeq} \det(I-zA) = 1-\sum_{j\ge 2} e_j(x_1,\dots,x_k) \,t[j-1]_t z^j.\eq
It follows that $$\det(I-zA:i,i) = 1 - \sum_{j\ge 2} e_j(x_1,\dots,\hat{x_i}, \dots, x_k) \,t[j-1]_t z^j,$$
where $ \hat{x_i} $ denotes deletion of $x_i$. Multiplying both sides by $x_i$ and summing over all $i\in [k]$ yields,

\begin{equation*} \sum_{i=1}^k x_i \det(I-zA:i,i) = \sum_{i=1}^k x_i - \sum_{j\ge 2} \sum_{i=1}^k x_i e_j(x_1,\dots,\hat{x_i}, \dots, x_k) \,t[j-1]_t z^j.
\end{equation*}
One can see that $$\sum_{i=1}^k x_i e_j(x_1,\dots,\hat{x_i}, \dots, x_k) = (j+1)e_{j+1}(x_1,\dots,x_k),$$ since both sides enumerate $(j+1)$-subsets of $[k]$ with a distinguished element.
Hence
$$ \sum_{i=1}^k x_i \det(I-zA:i,i) = e_1(x_1,\dots,x_k) - \sum_{j \ge 2} (j+1) e_{j+1}(x_1,\dots,x_k)  t[j-1]_tz^j.$$
Upon multiplying both sides by $z$, we see that the numerator of the right hand side of (\ref{deteq}) is 
$$e_1(x_1,\dots,x_k) z - \sum_{j \ge 3} j\, e_{j}(x_1,\dots,x_k) t[j-2]_t  z^j.$$
It therefore follows from (\ref{deteq}) and (\ref{elzeq}) that 
$$\sum_{n\ge 1} W^=_n(x_1,\dots,x_k,t) z^n = \frac{e_1(x_1,\dots,x_k) z - \sum_{j \ge 3} j\, e_{j}(x_1,\dots,x_k) t[j-2]_t  z^j}{1-\sum_{j\ge 2} e_j(x_1,\dots,x_k) \,t[j-1]_t z^j}.$$
The desired result (\ref{equaleq}) follows by taking the limit as $k$ goes to infinity.
\end{proof}

\begin{proof}[Proof of  (\ref{lesseq})] 

 It follows from (\ref{refineeq1}),  (\ref{goodtcarleq}), and (\ref{equaleq}) that
\begin{eqnarray} \label{lessgreateq1} & & \\
\nonumber \sum_{n \ge 1} W_n^<(\x,t) z^n + \sum_{n \ge 1} W_n^>(\x,t) z^n &=& \sum_{n \ge 1} W_n(\x,t) - \sum_{n \ge 1} W^=_n(\x,t)   \\ \nonumber
& =& {B(\x,t,z) \over D(\x,t,z)} , \end{eqnarray}
where 
\begin{eqnarray*} B(\x,t,z) &=& \sum_{i \ge 1} [i]_t e_i(\x) z^i -
(e_1(\x) z - \sum_{i \ge 2}  i t[i-2]_t \,\, e_i(\x) z^i) \\
&=& \sum_{i \ge 2} [i]_t e_i(\x) z^i + 
 \sum_{i \ge 2}  i t[i-2]_t \,\, e_i(\x) z^i \\
&=& \sum_{i \ge 2} ([i]_t+i t[i-2]_t) e_i(\x) z^i \,.
\end{eqnarray*}
It follows from (\ref{refineeq2}) and (\ref{tcirceq}) that
\begin{eqnarray} \label{lessgreateq2} & & \\ \nonumber t \sum_{n \ge 1}  W_n^<(\x,t) z^n + \sum_{n \ge 1} W_n^>(\x,t) z^n &=& \sum_{n \ge 1} \tilde W^{\ne}_n(\x,t)  z^n \\ \nonumber
& =& {C(\x,t,z) \over D(\x,t,z)} , \end{eqnarray}
where 
$$C(\x,t,z) = \sum_{i \ge 2} it[i-1] _t \,\,e_i(\x) z^i  . 
$$

By subtracting (\ref{lessgreateq1}) from (\ref{lessgreateq2}), we obtain
\begin{eqnarray*} (t-1)  \sum_{n \ge 1} W_n^<(\x,t) z^n &=& {C(\x,t,z)- B(\x,t,z) \over D(\x,t,z)} 
\\ &=& {\sum_{i\ge 2} ( it^{i-1} - [i]_t ) e_i(\x) z^i \over D(\x,t,z)} 
.\end{eqnarray*}
Note that 
\begin{eqnarray} \nonumber(t-1)(1+2t +3t^2 + \dots + (i-1)t^{i-2} ) &=& (i-1)t^{i-1} - [i-1]_t \\
\label{tminu1eq} &=& it^{i-1} - [i]_t
. \end{eqnarray}
Hence
$$  \sum_{n \ge 1} W_n^<(\x,t) z^n = {\sum_{i \ge 2}(1+2t +3t^2 + \dots + (i-1)t^{i-2}) \, e_i(\x) z^i \over D(\x,t,z)} $$
as desired.
\end{proof}

\begin{proof}[Proof of  (\ref{greateq})] Using the involution on the set of Smirnov words that reverses each word, we see that
$W_n^>(\x,t) = t^{n-1} W_n^<(\x,t^{-1})$.  It follows from this and~(\ref{lesseq}) that 
\begin{eqnarray*}\sum_{n \ge 1} W_n^>(\x,t) z^n &=& t^{-1} \sum_{n \ge 1} W_n^< (\x,t^{-1}) (tz)^n \\
&=& t^{-1} \frac{\sum_{i \ge 2} a_i(t^{-1}) e_i(\x)t^i z^i }{D(\x,t^{-1},tz)}\\
&=& \frac{\sum_{i \ge 2} b_i(t) e_i(\x)z^i }{D(\x,t^{-1},tz)}.
\end{eqnarray*}  Since $D(\x,t^{-1},tz) = D(\x,t,z)$, the result holds.
\end{proof}

Since  \bq \label{Deq} {1\over D(\x,t,z)} = \sum_{m \ge 0} \left(\sum_{i \ge 2} t[i-1]_t e_i(\x) z^i\right)^m, \eq we have the following consequence of Theorem~\ref{mainth}.

 \begin{cor} \label{lesgrcor}For all $n \ge 1$, the polynomials $W^<_n(\x,t)$ and $W^>_n(\x,t)$ are $e$-positive.\end{cor}
 
 Note that it follows from Theorem~\ref{mainth} that the  coefficient of $e_n(\x)$ in the $e$-basis expansion of $W^=_n(\x,t)$ is $-nt[n-2]_t$ if $n\ge 2$. Hence $W^=_n(\x,t)$ fails to be $e$-positive.  However, observe that the coefficient $c_\lambda(t)$ of  $e_\lambda(\x)$ is in $\N[t]$ if the smallest part of  $\lambda$ is $1$, and $- c_\lambda(t) \in \N[t]$ otherwise.

We obtain equivalent formulations of (\ref{lesseq}) and of (\ref{greateq})  by multiplying the numerators and denominators of the right hand sides of the equations by $1-t$.

\begin{cor} 
 We have
\begin{eqnarray}\label{lessWeeq} \sum_{n \ge 1} W_n^<(\x,t) z^n &=& \frac{(1-t)\frac{\partial}{\partial t}  \sum_{i\ge 2} [i]_t \, e_i(\x) z^i}{E(tz) - tE(z)}
 \\ \label{greatWeeq}
\sum_{n \ge 1} W_n^>(\x,t) z^n &=& \frac{(1-t) \sum_{i\ge 2} ( \sum_{j=1}^{i-1} (i-j) t^j)\, e_i(\x) z^i}{E(tz) - tE(z)}
\end{eqnarray} 
\end{cor}
\noindent where $E(z) := \sum_{n \ge 0} e_n(\x) z^n$.

We now use Theorem~\ref{mainth} to prove (\ref{noteqeq}) and (\ref{circfulleq}), which we restate here.

\begin{cor}   \label{noteqcor}We have,
\beq \sum_{n\ge 1} W^{\ne}_n(\x,t) z^n = \frac{ \sum_{i\ge 2}( [i]_t + it[i-2]_t) e_i(\x) z^i} {D(\x,t,z)}.\eeq
Consequently, $W^{\ne}_n(\x,t)$ is $e$-positive.
\end{cor}

\begin{proof} We use the facts that $W_n^{\ne}(\x,t) = W^{<}_n(\x,t) + W^{>}_n(\x,t)$ and  $a_i(t) + b_i(t) =  [i]_t + it[i-2]_t$.
\end{proof}

 \begin{cor} \label{circecor} We have,
\bq  \label{circWeeq} \sum_{n\ge 1} \tilde W_n(\x,t) z^n = \frac{\sum_{i\ge 1}  i t^{i-1} e_i(\x) z^i} {D(\x,t,z)}.
\eq
Consequently, $\tilde W_n(\x,t)$ is  $e$-positive.
\end{cor}

\begin{proof} 
We use the fact that
\bq \tilde W_n(\x,t) = t W^<_n(\x,t) + (W_n(\x, t) - W^<_n(\x,t)) \eq
and equations~(\ref{lesseq}), (\ref{tminu1eq}), and~(\ref{goodtcarleq}).
\end{proof}

\begin{cor} \label{circecor2} We have,
\bq \label{circWeeq2} \sum_{n\ge 1} \tilde W_n(\x,t) z^n =  \frac{(1-t)\frac{\partial} {\partial t} E(tz)}{E(tz)-tE(z)}.\eq
\end{cor}

In  \cite[Theorem 5.1]{sw1} (see also~\cite[Corollary C.5]{sw4}), it is observed  that (\ref{goodtcarleq}) implies that the Smirnov enumerator $W_n(\x,t)$ has a stronger property than $e$-positivity, namely  $e$-unimodality, and in \cite{e2,e1} the same is observed for   $\tilde W^{\ne}_n(\x,t)$ as a consequence of (\ref{tcirceq}).  Here we show that $W^{\ne}_n(\x,t)$ also has the stronger property, while $\tilde W_n(\x,t)$ does not.

A polynomial $A(t)=a_0 + a_1 t +\cdots  + a_{n} t^{n} \in \Q[t]$ is said to be {\em unimodal}\footnote{Note that our definition is nonstandard in that positivity of the coefficients is also required.} if 
$$0 \le a_0 \le a_1\le \cdots \le a_c \ge a_{c+1} \ge a_{c+2} \ge \cdots \ge a_n \ge 0,$$ for some $c$. Let $\Lambda_\Q$ denote the $\Q$-algebra of symmetric functions over $\x:=x_1,x_2,\dots$. Given $r,s \in \Lambda_\Q$, we say that $r \le_e s$ if $s-r$ is $e$-positive.  A polynomial $A(t)=a_0 + a_1 t +\cdots  + a_{n} t^{n} \in \Lambda_\Q[t]$ is said to be $e$-{\em unimodal} if
$$0 \le_e a_0 \le_e a_1 \le_e \cdots \le_e a_c \ge_e a_{c+1}\ge_e a_{c+2}\ge_e \cdots \ge_e a_n \ge_e 0,$$ for some $c$.   We say that $A(t)$ (over any coefficient ring) is {\em palindromic} with center of symmetry $\frac n 2$  if  $a_j = a_{n-j}$  for $0 \le j \le n$.  (Note that the center of symmetry is unique unless $A(t)$ is the zero polynomial, in which case every number of the form $\frac n 2$, where $n \in \N$, satisfies the definition of center of symmetry.)  It is easy to see that $A(t) \in \Lambda_{\Q}[t] $ is $e$-unimodal and palindromic with center of symmetry $c$  if and only if  all the coefficients in the $e$-expansion of $A(t)$ are unimodal and palindromic polynomials in $t$ with the same center of symmetry $c$; see \cite[Proposition B.3]{sw4}.

\begin{lemma} \label{uniprop} Let $(g_n(t))_{n\ge 2}$ be a sequence of polynomials in $\Q[t]$,
such that each $g_n(t)$   is  unimodal  and palindromic with center of symmetry $\frac {n+r} 2$, where 
$r \ge -2$ is a fixed integer.  
If $(G_n(\x,t))_{n \ge 2}$ is a sequence of polynomials in $\Lambda_\Q[t]$ that satisfies
$$\sum_{n \ge 2} G_n(\x,t) z^n = \frac {\sum_{n \ge 2} g_n(t) e_n(\x) z^n} {D(\x,t,z)},$$ 
where $D(\x,t,z)$ is defined in (\ref{Ddef}),
 then each $G_n(\x,t)$ is   $e$-unimodal and palindromic    with center of symmetry $\frac {n+r} 2$. \end{lemma}

 \begin{proof}
  We use Propositions~B.1 and~B.3 of \cite{sw4}. 
By (\ref{Deq}),   
\bq \label{Geq} G_n(\x,t) = \sum_{m \ge 1} \sum_{\scriptsize \begin{array}{c}  k_1,\dots, k_m \ge 2 \\ \sum_{i=1}^m k_i = n \end{array} }\!\! e_{k_1} \dots e_{k_m} t^{m-1}  g_{k_1}(t) \prod_{i=2}^{m}[k_i-1]_t .\eq  
For each nonzero $g_{k_1}(t)$, the polynomial 
$t^{m-1}g_{k_1}(t)  \prod_{i=2}^{m} [k_i-1]_t $ is a product of palindromic, unimodal polynomials.  Hence, the product is also  palindromic  and unimodal with center of symmetry equal to $$m-1+ \frac {k_1+r} 2 + \sum_{i=2}^m \frac {k_i-2} 2 = \frac {n+r} 2 .$$
Since each such product has the same center of symmetry, $G_n(\x,t)$ is palindromic and $e$-unimodal with center of symmetry $\frac {n+r} 2 $.
 \end{proof}
 
 \begin{cor}[of Corollary~\ref{noteqcor}]  For all $n \ge 2$, the Smirnov word enumerator $W^{\ne}_n(\x,t)$ is  $e$-unimodal and palindromic with center of symmetry $\frac {n-1} 2$.
\end{cor}

\begin{proof} 
Since  $[i]_t + it[i-2]_t$ is unimodal and palindromic  with center of symmetry $\frac{i-1} 2$, the result follows from Lemma~\ref{uniprop} 
\end{proof}

We note that although $W_n(\x,t)$, $W^{\ne}_n(\x,t)$, and $\tilde W^{\ne}(\x,t)$ are all $e$-unimodal and palindromic, this is not the case for $\tilde W_n(\x,t)$.  Indeed, it follows from Corollary~\ref{circecor} that
$$\tilde W_5(\x,t) = e_{4,1} t + (e_{2,2,1} + e_{4,1} + 2 e_{3,2}) t^2 + (e_{4,1} + 5 e_{3,2}) t^3 + (5e_5) t^4,$$
which is neither palindromic nor $e$-unimodal.  However, for certain partitions $\lambda$,  the coefficient of $e_\lambda$ is a palindromic and unimodal polynomial in $t$.

\begin{cor}[of Corollary~\ref{circecor}] For $\lambda \vdash n$, let $c_\lambda(t)$ be the coefficient of $e_\lambda$ in the expansion of $\tilde W_n(\x,t)$ in the elementary symmetric functions.
If  $\lambda=(\lambda_1\ge \dots \ge \lambda_{k-1} \ge 1)$ then $$c_\lambda(t) = t^{k-1}\prod_{i=1}^{k-1} [\lambda_i-1]_t,$$ which is palindromic and unimodal.  If $\lambda = j^k$ then  $$c_\lambda(t) = j t^{j+k-2} [j-1]_t^{k-1},$$ which is palindromic and unimodal.
\end{cor}

\section{Chromatic quasisymmetric function of  the  cycle} \label{conssec}

In this section, we discuss the connection between the  Smirnov word enumerators and  the chromatic quasisymmetric functions introduced by Shareshian and the second named author in \cite{sw5,sw4}.   We use results of the previous section to provide an example of an e-positive chromatic quasisymmetric function not covered by the  refinement of the Stanley-Stembridge $e$-positivity conjecture appearing in \cite{sw5,sw4} or its directed graph extension appearing in \cite{e2,e1}.

For any graph $G = ([n],E)$,  define the {\em chromatic quasisymmetric function}\footnote{In the definition given in \cite{sw4} ascents are counted instead of descents, but in the  case that $X_G(\x,t)$ is symmetric, it is shown in \cite[Corollary 2.7]{sw4}  that the definitions are equivalent.} of  $G$  as 
$$X_G(\x,t) := \sum_{\kappa \in \mathcal C(G)} t^{\des(\kappa)} x_{\kappa(1)} x_{\kappa(2)} \cdots x_{\kappa(n)},$$
where 
$\mathcal C(G)$ is the set of proper colorings $\kappa:[n] \to \PP$ of $G$  and 
\bq \label{ascdef}  \des(\kappa) :=| \{\{i,j\} \in E : i < j \mbox{ and  } \kappa(i) > \kappa(j) \}|.\eq
When $t=1$, $X_G(\x,t)$ is Stanley's chromatic symmetric function $X_G(\x)$. Note that $X_G(\x,t)$ is a polynomial in $t$ whose coefficients are quasisymmetric functions.
 We view $G$ as a labeled graph  and note that  the definition of $X_G(\x,t)$ depends on the vertex labeling, and not just on the isomorphism class of $G$, as is the case for  $X_G(\x)$.

Since Smirnov words of length $n$ can be viewed as proper colorings of the   naturally labeled path $$P_n:=([n],\{\{i,i+1 \}: i \in [n-1]\}),$$ it follows that   $$W_n(\x,t)=X_{P_n}(\x,t).$$

The Smirnov word enumerator
$\tilde W_n^{\ne}(\x,t)$ can also be viewed as a chromatic quasisymmetric function, but in the  more general sense considered by the first named author    \cite{e2,e1},  in which labeled graphs are replaced by directed graphs and  the definition of $\des(\kappa)$ given  in (\ref{ascdef}) is  replaced by
$$\des(\kappa) := |\{(i,j) \in E : \kappa(i) > \kappa(j) \}|.$$  Labeled graphs can be viewed as directed graphs  by orienting each edge from smaller vertex to larger vertex; so the digraph version of chromatic quasisymmetric function is more general than the labeled graph version.  One can see that
$$\tilde W_n^{\ne}(\x,t) = X_{\overrightarrow C_n}(\x,t),$$
where $\overrightarrow C_n$ is the directed cycle defined by
$$\overrightarrow C_n:= ([n], \{(i,i+1 ): i \in [n-1]\}\cup \{(n,1)\}).$$

The longstanding Stanley-Stembridge conjecture \cite{st4}  asserts that $X_G(\x)$ is $e$-positive when $G$ is the incomparability graph of a $(3+1)$-free poset.  
In \cite{gp}, Guay-Paquet proves that if the Stanley-Stembridge conjecture holds for  incomparability graphs of posets that are both $(3+1)$-free and $(2+2)$-free  (known as {\em unit interval graphs}) then it holds in general. 
  
   In \cite{sw4} Shareshian and the second named author show that $X_G(\x,t)$ is symmetric when $G$ is a  unit interval graph with a certain natural labeling; these are called {\em natural unit interval graphs}.  They  also conjecture that $X_G(\x,t)$ is $e$-positive and $e$-unimodal when $G$ is a natural unit interval graph.  The path $P_n$ is an example of a natural unit interval graph for which the conjecture  holds  since $X_{P_n}(\x,t) = W_n(\x,t)$.   The symmetry result and $e$-positivity conjecture of \cite{sw4} are generalized in \cite{e2,e1} to a class of directed graphs called {\em indifference digraphs} in \cite{e1}.   With the view that a labeled graph is an acyclic digraph, the natural unit interval graphs form the class of acyclic indifference digraphs. The directed cycle $\overrightarrow C_n$ is an example of an indifference digraph for which the extended conjecture  holds since $X_{\overrightarrow C_n}(\x,t) = \tilde W^{\ne}_n(\x,t)$. 

Here we consider the labeled cycle
\bq \label{cycdef} C_n :=([n], \{\{i,i+1 \}: i \in [n-1]\}\cup \{\{1,n\}\}). \eq
If we view $C_n$ as a directed graph (by orienting its edges from smaller vertex to larger vertex), we get a directed graph
that is identical to the directed cycle $\overrightarrow C_n$ except for the edge $(n,1)$ in  $\overrightarrow C_n$ which is oriented as $(1,n)$ in $C_n$.   For $n \ge 4$, the labeled cycle $C_n$ is not  a natural unit interval graph, nor is it an  indifference digraph.  Nevertheless, since 
\bq \label{labcycWeq} X_{C_n}(\x,t) = W_n^<(\x,t) + t W_n^>(\x,t) ,\eq 
it follows from Corollary ~\ref{lesgrcor} that  $X_{C_n}(\x,t)$ is  symmetric and $e$-positive.  This shows that the class of labeled graphs with $e$-positive chromatic quasisymmetric function is strictly larger than the class of natural unit interval graphs, and the class of digraphs with $e$-positive chromatic quasisymmetric function is strictly larger than the class of  indifference digraphs.

Next we address the question of $e$-unimodality of $X_{C_n}(\x,t)$.   From Theorem~\ref{mainth}  and equation (\ref{labcycWeq}), we obtain the next result.  For $n >0$, let $$ [-n]_t := \frac{t^{-n}-1}{t-1} = - t^{-n}[n]_t.$$
 \begin{cor}[of Theorem~\ref{mainth}]  \label{cyclecor} We have,
 \bq  \label{labcyceq}\sum_{n\ge 2} X_{C_n}(\x,t) z^n = \frac{   \sum_{i\ge 2}( [2]_t[i]_t+ it^2[i-3]_t) e_i(\x) z^i} {D(\x,t,z)},\eq
 where $D(\x,t,z)$ is defined in (\ref{Ddef}).
 \end{cor}

\begin{thm} Let $n \ge 2$.
\begin{enumerate} 
\item If $n$ is odd,  $X_{C_n}(\x,t)$ is $e$-unimodal  and palindromic with center of symmetry $\frac n 2$. 
\item If $n$ is even, \begin{enumerate}
\item  $X_{C_n}(\x,t)$ is  $e$-positive and palindromic  with center of symmetry $\frac n 2$, but is not $e$-unimodal. 
\item  $X_{C_n}(\x,t) + t^{\frac n 2} e_{2^{\frac n 2}}(\x)$ is $e$-unimodal and palindromic 
with center of symmetry $\frac n 2$. \end{enumerate} \end{enumerate}
\end{thm}

\begin{proof} Let $U_n(\x,t)$ and $V_n(\x,t)$ be defined respectively by 
$$\sum_{n \ge 2} U_n(\x,t) z^n = \frac{ ( [2]_t[2]_t+ 2t^2[2-3]_t) e_2(\x) z^2} {D(\x,t,z)}  $$ 
and $$\sum_{n \ge 2} V_n(\x,t) z^n = \frac{   \sum_{i\ge 3}( [2]_t[i]_t+ it^2[i-3]_t) e_i(\x) z^i} {D(\x,t,z)} .$$
Then $X_{C_n}(\x,t) = U_n(\x,t) + V_n(\x,t)$. 

We have,
$$\sum_{n \ge 2} U_n(\x,t) z^n = \frac {(1+t^2)e_2(\x) z^2}{D(\x,t,z)}.$$
 It follows from (\ref{Deq}) that 
\bq \label{Ueq} U_n(\x,t) = \sum_{m \ge 1} \sum_{\scriptsize \begin{array}{c}  k_2,\dots, k_m \ge 2 \\ \sum_{i=2}^m k_i = n-2 \end{array} }\!\! e_{2}e_{k_2} \dots e_{k_m} t^{m-1}  (1+t^2) \prod_{i=2}^{m}[k_i-1]_t .\eq

Note that for any $k \ge 3$ 
$$(1+t^2)[k]_t = 1+ t + 2t^2 + \dots + 2t^{k-1} +t^k + t^{k+1},$$
and for $k=2$,
$$(1+t^2)[k]_t = 1+t+t^2+t^3.$$  In either case, $(1+t^2)[k]_t $ is 
unimodal and palindromic with center of symmetry $\frac {k+1}2$.
 
 We now use Propositions B.1 and B.3 of \cite{sw4}.  Consider the term of the right side of (\ref{Ueq}) corresponding to the $(m-1)$-tuple $(k_2,\dots,k_{m})$. 
If $k_j \ge 3$ for  some $j\ge 2$ then since $(1+t^2)[k_j-1]_t$  is unimodal and palindromic, $t^{m-1}(1+t^2) \prod_{i=2}^m[k_i-1]_t $ is a product of unimodal and palindromic polynomials.  Hence  $t^{m-1}(1+t^2) \prod_{i=2}^m[k_i-1]_t $ is unimodal and  palindromic with center of symmetry
$$m-1 + \frac{k_j}2+ \sum_{\substack {i = 2 \\ i \ne j}}^m \frac{k_i-2}{2} = \frac n 2.$$ 
It follows that if $\lambda$ has a part of size at least $3$ then the coefficient of  $e_\lambda$ in $U_n(\x,t)$  is palindromic and unimodal with center of symmetry $\frac n 2$.
If $\lambda$ does not have a part of size at least $3$ then all the parts must be $2$, which means that $n$ is even.
Hence if $n$ is odd then $U_n(\x,t)$ is $e$-unimodal and palindromic with center of symmetry $\frac{n}{2}$.  

Now if $\lambda$ does not have a part of size at least $3$ then $\lambda = 2^m$, where $n= 2m$. By (\ref{Ueq}), the coefficient of $e_\lambda$ in $U_n(\x,t)$ is $t^{m-1}(1+t^2)$.  Hence if $\lambda$ does not have a part of size at least $3$ then coefficient of $e_\lambda$ in  $U_n(\x,t)+  t^m e_{2^m}$ is unimodal and palindromic with center of symmetry $m=\frac n 2$.  From the argument in the previous paragraph, the same is true if  $\lambda$ has a part of size at least $3$.  It follows that 
 if $n$ is even, $U_n(\x,t) + t^m e_{2^m}$ is $e$-unimodal and palindromic with center of symmetry $m=\frac n 2$. 

It follows from Lemma~\ref{uniprop} that $V_n(\x,t)$ is also  palindromic and $e$-unimodal with center of symmetry $\frac n 2$. Since $X_{C_n}(\x,t)= U_n(\x,t) + V_n(\x,t)$, Parts (1) and (2b) hold.  Palindromicity of $X_{C_n}(\x,t)$ in the even case follows from Part (2b).  The assertion in Part (2a) that $X_{C_n}(\x,t)$ is not $e$-unimodal in the even case follows from the fact the coefficient of $e_{2^m}(\x)$ is $t^{m-1}(1+t^2)$, which is not unimodal.  
\end{proof}

\section{Expansion in the power sum symmetric functions}

Let $A_n(t)$ be the Eulerian polynomial defined in (\ref{combeuldef}) for $n \ge 1$ and let $$A_0(t):= t^{-1}.$$  Also let  $\omega$ be the standard involution on $\Lambda_\Q$ taking the elementary symmetric function $e_n$ to the complete homogeneous symmetric function $h_n$,
 and let $$H(z) := \sum_{n\ge 0} h_n(\x) z^n, $$ where $h_n(\x)$ is the complete homogenous symmetric function of degree $n$.   
Stembridge  \cite[Proposition 3.3]{stem1}  proves that 
\bq \label{stemeq} 1+ \sum_{n\ge 1} \sum_{\lambda \vdash n} \left (A_{l(\lambda)}(t) \prod_{i=1}^{l(\lambda)} [\lambda_i]_t \right )\frac{p_\lambda(\x)}{z_\lambda} z^n=\frac{ (1-t)H(z) }{H(tz) - tH(z)} \eq
where $\lambda = (\lambda_1 \ge \lambda _2 \ge \dots \ge \lambda_{l(\lambda)})$  and
$$z_\lambda := \prod_{i \ge 1} i^{m_i} m_i! $$ if $\lambda$  has $m_i$ parts of size $i$ for each $i$. 
 By combining this with (\ref{goodtcarleq2}) one obtains the expansion of  $\omega W_n(\x,t)$ in the power sum symmetric functions given in ({\ref{powereq}).

In this section we derive  power sum expansions for the other Smirnov word enumerators.    We will use the $e$-expansion for $\tilde W_n(\x,t)$ obtained in Section~\ref{esec} to obtain the following result, which expresses the coefficients of $\frac{p_\lambda(\x)}{z_\lambda}$ in the power sum expansion of $\omega \tilde W_n(\x,t)$ as a polynomial in $t$ with positive integer coefficients.

\begin{thm} \label{pcircth} For all $n \ge 1$,
\begin{equation} \label{pcirceq} \omega \tilde W_n(\x,t) = \sum_{\lambda \vdash n} \left (A_{l(\lambda)-1}(t) \sum_{i=1}^{l(\lambda)} \lambda_i t^{\lambda_i} \prod_{j \in [l(\lambda)]\setminus\{i\}} [\lambda_j]_t\right ) \frac{p_\lambda(\x)}{z_\lambda} ,\end{equation}
where $\lambda = (\lambda_1 \ge \lambda _2 \ge \dots \ge \lambda_{l(\lambda)})$. 
\end{thm}

We will need the following Lemma, which is implicit in the proof of (\ref{stemeq}) in \cite{stem1}.  We include the proof for the sake of completeness.
\begin{lemma} \label{stemlem} For all $k \ge 1$,
\bq \label{Hfraceq} \left ( \frac{H(z)}{H(tz)}\right )^{k}= 1+\sum_{n\ge 1} \sum_{\lambda \vdash n} \left (k^{l(\lambda)} \prod_{i=1}^{l(\lambda)} (1-t^{\lambda_i}) \right) \frac{p_\lambda(\x)}{z_\lambda} z^n .\eq
\end{lemma}

\begin{proof} For each $k \ge 1$, let $\varphi_k:\Lambda_\Q[z] \to \Lambda_\Q[t,z] $ be the algebra homomorphism determined  by
$$\varphi_k(p_r(\x) ) = k(1-t^r) p_r(\x) $$ for all $ r \ge 1$.  
Since $H(z) = 1+\sum_{n\ge 1}\sum_{\lambda \vdash n} \frac{p_\lambda}{z_\lambda} z^n$,
$$\varphi_k(H(z)) =1+ \sum_{n\ge 1}\sum_{\lambda \vdash n} \left (k^{l(\lambda)}\prod_{i=1}^{l(\lambda)} (1-t^{\lambda_i}) \right )\frac{p_\lambda}{z_\lambda} z^n.$$

To complete the proof  we show that $\varphi_k(H(z)) $ is equal to the left hand side of (\ref{Hfraceq}). We use the fact that
\bq \label{hpeq} H(z) = \exp(\sum_{r \ge 1} \frac{p_r(\x)}r z^r)\eq  to obtain
\begin{eqnarray*} \varphi_k(H(z)) &=& \exp\left( \sum_{r \ge 1} \frac{\varphi_k(p_r)}{r} z^r\right)
\\ &=& \exp \left( \sum_{r \ge 1} \frac{k(1-t^r) p_r}{r} z^r\right)
\\ &=& \left (\frac{\exp \left( \sum_{r \ge 1} \frac{ p_r}{r} z^r\right)}{ \exp \left( \sum_{r \ge 1} t^r \frac{p_r}{r} z^r\right)}\right )^k
\\&=& \left ( \frac{H(z)}{H(tz)} \right )^k.
\end{eqnarray*}
\end{proof}

\begin{proof}[Proof of Theorem~\ref{pcircth}] For each $\lambda \vdash n$, set $$c_\lambda(t) := A_{l(\lambda)-1}(t) \sum_{i=1}^{l(\lambda)} \lambda_i t^{\lambda_i} \prod_{j \in [l(\lambda)]\setminus\{i\}} [\lambda_j]_t.$$  We will prove that
\bq \label{equiv42eq} \sum_{n\ge 1} \sum_{\lambda \vdash n} c_\lambda(t) \frac{p_\lambda(\x)}{z_\lambda} z^n =  \frac{(1-t)\frac{\partial} {\partial t} H(tz)}{H(tz)-tH(z)},\eq which by Corollary~\ref{circecor2} is equivalent to (\ref{pcirceq}).

We have 
\begin{eqnarray*} c_\lambda(t) &=& (t-1) \frac{A_{l(\lambda)-1}(t)}{(t-1)^{l(\lambda)}} \sum_{i=1}^{l(\lambda)} \lambda_i t^{\lambda_i} \prod_{j \in [l(\lambda)]\setminus\{i\}} (t^{\lambda_j}-1) 
\\ &=& (t-1) (-1)^{l(\lambda)} \frac{A_{l(\lambda)-1}(t)}{(1-t)^{l(\lambda)}} \, t \,\frac{d}{dt}\left ( \prod_{j=1}^{l(\lambda)} (t^{\lambda_j}-1) \right) \\
&=& (t-1)  \frac{tA_{l(\lambda)-1}(t)}{(1-t)^{l(\lambda)}} \, \frac{d}{dt}\left ( \prod_{j=1}^{l(\lambda)} (1-t^{\lambda_j}) \right).
\end{eqnarray*}

For the case $l(\lambda)=1$, we have
$$c_{(n)}(t)= n t^{n-1}.$$
For the case $l(\lambda)>1$, we use the classical identity\footnote{This is Euler's original definition of Eulerian polynomial.} (see \cite[Proposition 1.4.4 and equation (1.36)]{st5}),  
\bq \label{eulerdef} \frac {t A_{m-1}(t) } {(1-t)^m} =\sum_{k \ge 1} k^{m-1} t^{k},\eq
for all $m > 1$.  This yields
\begin{eqnarray*} c_\lambda(t) &=& (t-1) \sum_{k \ge 1} k^{l(\lambda)-1} t^{k} \,\frac{d}{dt}\left ( \prod_{j=1}^{l(\lambda)} (1-t^{\lambda_j}) \right) 
\\ &=& (t-1) \sum_{k \ge 1}  \frac{t^{k}}{k} \,\frac{d}{dt}\left ( k^{l(\lambda)} \prod_{j=1}^{l(\lambda)} (1-t^{\lambda_j}) \right). 
\end{eqnarray*}
It follows  that
\begin{equation} 
\label{pexpeq} 
\sum_{n\ge 1} \sum_{\lambda \vdash n}  c_\lambda(t) \frac{p_\lambda(\x)}{z_\lambda} z^n 
=  \sum_{n \ge 1} nt^{n-1} \frac{p_{n}(\x)}{n} z^n 
+ (t-1) \sum_{k\ge 1} \frac{t^{k}}{k} \frac{\partial}{\partial t} U_k(\x,t,z)\end{equation}
where 
$$U_k(\x,t,z):=\sum_{n\ge 1} \sum_{\scriptsize\begin{array}{c} \lambda \vdash n \\ l(\lambda) >1 \end{array}}  k^{l(\lambda)} \prod_{j=1}^{l(\lambda)} (1-t^{\lambda_j}) \, \frac{p_\lambda(\x)}{z_\lambda} z^n . $$

Note that  the first summation on the right hand side of (\ref{pexpeq}) can be expressed as \begin{eqnarray}\nonumber \sum_{n \ge 1} nt^{n-1} \frac{p_{n}}{n} z^n &=&\frac{\partial}{\partial t}  \sum_{n \ge 1} t^n \frac{p_{n}}{n} z^n 
\\ \nonumber &=&\frac{\partial}{\partial t} \ln H(tz)
\\ \label{firstsumeq} &=& \frac{\frac{\partial }{dt} (H(tz))}{H(tz)}
\end{eqnarray}
with the second equality following from (\ref{hpeq}).

To evaluate the second summation on the right hand side of (\ref{pexpeq}), we use  Lemma~\ref{stemlem} to obtain
\begin{eqnarray*} U_k(\x,t,z) &=& \sum_{n\ge 1} \sum_{ \lambda \vdash n}  k^{l(\lambda)} \prod_{j=1}^{l(\lambda)} (1-t^{\lambda_j}) \frac{p_\lambda(\x)}{z_\lambda} z^n - \sum_{n\ge 1} k(1-t^n)  \frac{p_n(\x)}{n} z^n 
\\ &=& \left (\frac{H(z)}{H(tz)} \right  )^{k}-1- k\sum_{n\ge 1} (1-t^n)  \frac{p_n(\x)}{n} z^n  . \end{eqnarray*} 
By (\ref{firstsumeq}),
\begin{eqnarray*}\frac{\partial}{\partial t} U_k(\x,t,z) &=& k  \left (\frac{H(z)}{H(tz)} \right  )^{k-1} \frac{\partial }{\partial t} \left (\frac{H(z)}{H(tz)} \right )  + k \sum_{n\ge 1} t^{n-1}  p_n(\x) z^n
\\ &=& -k \left ( \left (\frac{H(z)}{H(tz)} \right )^{k} \frac{\frac{\partial }{\partial t} (H(tz))}{H(tz)} 
-\frac{\frac{\partial }{\partial t} (H(tz))}{H(tz)}
\right ) 
\\&=& - k\left ( \left (\frac{H(z)}{H(tz)} \right )^{k}-1\right ) \frac{\frac{\partial }{\partial t} (H(tz))}{H(tz)} .
\end{eqnarray*}
Hence the second summation is 
$$
-\frac{\frac{\partial }{\partial t} (H(tz))}{H(tz)} \sum_{k\ge 1} t^{k}\left (\left (\frac{H(z)}{H(tz)} \right )^{k}-1 \right )
.$$
Plugging this and (\ref{firstsumeq}) into (\ref{pexpeq}) yields
\begin{eqnarray*} \sum_{n\ge 1} \sum_{\lambda \vdash n}  c_\lambda(t) \frac{p_\lambda(\x)}{z_\lambda} z^n  &=&\frac{\frac{\partial }{\partial t} (H(tz))}{H(tz)} 
\\ & &+(1-t) \frac{\frac{\partial }{\partial t} (H(tz))}{H(tz)} \sum_{k\ge 1} t^{k}\left (\left (\frac{H(z)}{H(tz)} \right )^{k}-1 \right )
\\ &=&  \frac{\frac{\partial }{\partial t} (H(tz))}{H(tz)} \left (1+ (1-t) \sum_{k\ge 1} t^{k} \left (\left (\frac{H(z)}{H(tz)} \right )^{k}-1 \right ) \right ) \end{eqnarray*}

 \vspace*{-.7in}\begin{eqnarray*} \phantom {\sum_{n\ge 1} \sum_{\lambda \vdash n}  c_\lambda(t) \frac{p_\lambda(\x)}{z_\lambda} z^n}  &\phantom{=}& \phantom{\frac{\frac{\partial }{dt} (H(tz))}{H(tz)}}
\\
 &=&  \frac{\frac{\partial }{dt} (H(tz))}{H(tz)} \left ( 1- (1-t) \sum_{k \ge 0} t^{k} + (1-t) \sum_{k\ge 0} t^{k} \left (\frac{H(z)}{H(tz)} \right )^{k}  \right )
\\& =& \frac{\frac{\partial }{\partial t} (H(tz))}{H(tz)} (1-t) \frac 1 {1- t\frac{H(z)}{H(tz)}}
\\ &=&  \frac{(1-t)\frac{\partial} {\partial t} H(tz)}{H(tz)-tH(z)},
 \end{eqnarray*}
 which establishes (\ref{equiv42eq}).
\end{proof}

The following result expresses the coefficients of $\frac{p_\lambda(\x)}{z_\lambda}$ in the power sum expansion of $W^{<}_n(\x,t)$ as a polynomial in $t$ with positive integer coefficients.

\begin{thm} \label{powlesth} For all $n \ge 1$,
$$\omega W_n^{<}(\x,t) =  \sum_{\lambda \vdash n}\frac{d }{dt} \left(  tA_{l(\lambda)-1}(t) \prod_{i=1}^{l(\lambda)}[\lambda_i]_t \right)\frac{p_\lambda(\x)}{z_\lambda}. $$
\end{thm}

\begin{proof} Let $c^{<}_\lambda(t)$ be the coefficient of $ (z_\lambda)^{-1}p_\lambda(\x)$ in $\omega W^{<}_n(\x,t)$.  By (\ref{refineeq1}) and (\ref{refineeq2}),  $\omega W_n^{<}(\x,t)=(t-1)^{-1} (\omega \tilde W_n(\x,t) -\omega W_n(\x,t))$.  Hence from~(\ref{pcirceq}) and~(\ref{powereq}),  we obtain,
\begin{eqnarray*}c^{<}_\lambda(t) &=& 
 (t-1)^{-1} \left (A_{l(\lambda)-1}(t)  \sum_{i=1}^{l(\lambda)}  \lambda_i t^{\lambda_i} \prod_{j \in [\lambda(l)]\setminus \{i\}} [\lambda_j]_t  
 - A_{l(\lambda)}(t) \prod_{ i=1}^{l(\lambda)} [\lambda_i]_t \right)
 \\ &=&   
  \frac{A_{l(\lambda)-1}(t)} {(t-1)^{l(\lambda)}}  \sum_{i=1}^{l(\lambda)}  \lambda_i t^{\lambda_i} \prod_{j \in [\lambda(l)]\setminus \{i\}} (t^{\lambda_j}-1 ) -\frac{A_{l(\lambda)}(t)} {(t-1)^{l(\lambda)+1}} \prod_{ i=1}^{l(\lambda)}(t^{\lambda_i}-1)
 \\ &=& \frac{tA_{l(\lambda)-1}(t)} {(t-1)^{l(\lambda)}} \,\,\frac{d}{dt} (\prod_{ i=1}^{l(\lambda)}(t^{\lambda_i}-1))  - \frac{A_{l(\lambda)}(t)} {(t-1)^{l(\lambda)+1}} \prod_{ i=1}^{l(\lambda)}(t^{\lambda_i}-1) . 
 \end{eqnarray*}
 From (\ref{eulerdef}), one can see that
 $$\frac{d}{dt} \left (\frac{tA_{l(\lambda)-1}(t)} {(1-t)^{l(\lambda)}}\right ) = \frac{A_{l(\lambda)}(t)} {(1-t)^{l(\lambda)+1}} .$$
 Hence, 
 \begin{eqnarray*} c^{<}_{\lambda}(t) &=& \frac{tA_{l(\lambda)-1}(t)} {(t-1)^{l(\lambda)}} \,\,\frac{d}{dt} (\prod_{ i=1}^{l(\lambda)}(t^{\lambda_i}-1))  + \frac{d}{dt}\left(\frac{tA_{l(\lambda)-1}(t)} {(t-1)^{l(\lambda)}} \right ) \prod_{ i=1}^{l(\lambda)}(t^{\lambda_i}-1)  
 \\ &=& \frac{d}{dt}\left(\frac{tA_{l(\lambda)-1}(t)} {(t-1)^{l(\lambda)}}  \prod_{ i=1}^{l(\lambda)}(t^{\lambda_i}-1) \right )
 \\ &=& \frac{d }{dt} \left(  tA_{l(\lambda)-1}(t) \prod_{i=1}^{l(\lambda)}[\lambda_i]_t \right).
 \end{eqnarray*}
\end{proof}

\begin{cor} \label{powlesgrcor} For $\lambda \vdash n$,  let $c^{<}_\lambda(t)$ be the coefficient  of $z_\lambda^{-1}p_\lambda(\x)$ in the power sum expansion of $\omega W_n^{<}(\x,t)$ and $c^{>}_\lambda(t)$ be the coefficient  of $z_\lambda^{-1}p_\lambda(\x)$ in the power sum expansion of $\omega W_n^{>}(\x,t)$.  
 If  $(a_0,a_1,\dots, a_{n-1})$ satisfies
\bq \label{adef} tA_{l(\lambda)-1}(t) \prod_{i=1}^{l(\lambda)}[\lambda_i]_t = \sum_{i=0}^{n-1} a_i t^i\eq
then $$c^{<}_\lambda(t) = \sum_{i=0}^{n-2} (i+1)a_{i+1} t^i \,\,\mbox{ and } \,\,
c^{>}_\lambda(t) = \sum_{i=1}^{n-1} (n-i) a_{n-i} t^i.$$
Consequently for $\lambda = (n)$,
\bq \label{npowerlesgreq} c^{<}_{(n)}(t) = \sum_{i=0}^{n-2} (i+1) t^i \,\,\mbox { and }\,\, c^{>}_{(n)}(t) = \sum_{i=1}^{n-1} (n-i) t^i.\eq 
\end{cor}

\begin{proof}We use the fact that $W^{>}_n(\x,t) = t^{n-1} W^{<}_n(\x,t^{-1})$.
\end{proof}

One can use Corollary~\ref{powlesgrcor} to expand the other Smirnov word enumerators in the power sum basis. For instance, one can recover the expansion given in (\ref{powercycneeq}), which we restate now.

\begin{cor}[Ellzey \cite{e1}] \label{Ellth} For $\lambda \vdash n$,  let $\tilde c^{\ne}_\lambda(t)$ be the coefficient  of $z_\lambda^{-1}p_\lambda(\x)$ in the power sum expansion of $\omega \tilde W_n^{\ne}(\x,t)$.  Then $$\tilde c^{\ne}_\lambda(t) = \begin{cases} nt  A_{l(\lambda)-1}(t) \prod_{i=1}^{l(\lambda)} [\lambda_i]_t &\mbox{ if } l(\lambda) >1
\\ nt[n-1]_t &\mbox{ if } l(\lambda) =1 .\end{cases} $$
\end{cor}

\begin{proof} We have 
\begin{eqnarray*} \tilde c^{\ne}_{\lambda}(t) &=& tc^{<}_{\lambda}(t) + c^{>}_{\lambda}(t)
\\&=&\sum_{i=1}^{n-1} i a_i  t^i + \sum_{i=1}^{n-1} (n-i) a_{n-i} t^i ,
\end{eqnarray*}
where the $a_i$ are as in (\ref{adef}).

Now let $l(\lambda) >1$.  We claim that $a_i = a_{n-i}$ for all $i \in [n-1]$.  Indeed,
it is well known that the Eulerian polynomials are palindromic and unimodal.   Clearly the same is true for each $[\lambda_i]_t$.  Since the product of palindromic, unimodal polynomials is palindromic, unimodal (see e.g. \cite[Proposition~B.1]{sw4}), $\sum_{i=0}^{n-1} a_i t^i$ is palindromic (and unimodal). Note that  $a_0 =0$ when $l(\lambda) >1$ and $a_1, a_{n-1} \ne 0$. Hence the claim holds.
It follows that 
\begin{eqnarray*} \tilde c^{\ne}_\lambda(t) &=& \sum_{i=1}^{n-1} i a_i  t^i + \sum_{i=1}^{n-1} (n-i) a_{i} t^i
\\&=& \sum_{i=1}^{n-1} n a_i t^i
\\ &=& n tA_{l(\lambda)-1}(t) \prod_{i=1}^{l(\lambda)}[\lambda_i]_t .
\end{eqnarray*}

The case $l(\lambda) =1$ follows immediately from (\ref{npowerlesgreq}).
\end{proof}

For the Smirnov word enumerator $W^{\ne}_n(\x,t)$ and the chromatic quasisymmetric function $X_{C_n}(\x,t)$, the formulas for the expansion coefficients in the power sum basis that follow from Corollay~\ref{powlesgrcor} do not seem to reduce to  simple  formulas  except when $\lambda = (n)$.  We have the following result in this case.

\begin{cor} \label{otherpowercor}The coefficient of   $n^{-1}p_n(\x)$ in the power sum expansion of $\omega W^{\ne}_n(\x,t)$ is
$ [n]_t + nt[n-2]_t$ and in the power sum expansion of $\omega X_{C_n}(\x,t)$ is 
$[2]_t[n]_t+ nt^2[n-3]_t$.
\end{cor}

\begin{remark} Corollary~\ref{otherpowercor} also follows from Corollaries~\ref{noteqcor} and~\ref{cyclecor} since
the coefficient of $h_n$ in the $h$-expansion of a symmetric function equals the coefficient of  $n^{-1} p_n$ in the power sum expansion.
\end{remark}

\section{Expansion in the fundamental quasisymmetric functions}

In \cite{sw4}, Shareshian and the second named author derive, for all labeled incomparability graphs, an expansion of the chromatic quasisymmetic function in  Gessel's basis of fundamental quasisymmetric functions, and the first named author does the same for 
all directed graphs in  \cite{e1}.  The expansion formula in \cite[Theorem 3.1]{sw4} applied to  $X_{P_n}(\x,t) = W_n(\x,t) $  is given in  (\ref{Feq}) 
below. (A different expansion formula is  obtained by applying the formula  in \cite{e1}.)  Here we give analogous 
expansions for $ W^<_n(\x,t)$,  $W^<_n(\x,t)$, and $\omega \tilde W_n(\x,t)$.  These expansions immediately yield 
expansion formulas for  the chromatic quasisymmetric functions  $X_{\overrightarrow C_n}(\x,t)$, and $X_{C_n}(\x,t)$, which are different  from the ones obtained by applying the formula in \cite{e1}.

For $n \ge 1$ and $S \subseteq [n-1]$, let $D(S)$ be the set of all functions $f:[n] \rightarrow \PP$ such that
\begin{itemize}
\item $f(i) \geq f(i+1)$ for all $i \in [n-1]$, and \item $f(i)>f(i+1)$ for all $i \in S$.
\end{itemize}
The  {\em fundamental quasisymmetric function} associated with $S \subseteq [n]$ is defined as\footnote{This is  nonstandard notation for Gessel's fundamental quasisymmetric function.  Our $F_{n,S}$ is equal to $L_{\alpha(S)}$ in \cite{st3}, where $\alpha(S)$ is the reverse of the composition associated with $S$.}
\[
F_{n,S}(\x):=\sum_{f \in D(S)}x_f,
\]
where $x_f := x_{f(1)}x_{f(2)} \cdots x_{f(n)}$.
In fact, 
the set $\{F_{n,S}:S \subseteq [n-1]\}$ is a basis for the vector space of homogeneous quasisymmetric functions of degree $n$  (see \cite[Proposition 7.19.1]{st3}).

For $\sigma \in \sg_n$, define 
$$\Des_{\ge 2}(\sigma) := \{i \in [n-1] : \sigma(i) -\sigma(i+1) \ge 2 \}$$
and 
$$\Asc_{\ge 2}(\sigma) := \{i \in [n-1] : \sigma(i+1) -\sigma(i) \ge 2 \}.$$

The expansion formula in \cite[Theorem 3.1]{sw4} applied to $X_{P_n}(\x,t) = W_n(\x,t) $ yields \bq \label{Feq} \omega W_n(\x,t) = \sum_{\sigma \in \sg_n} t^{\des(\sigma)}F_{n, \Des_{\geq 2}(\sigma^{-1})}(\x).\eq
Now we give analogous expansions. 

 \begin{thm}  \label{Fth} For all $n \ge 1$, 
\bq \label{Flesseq} \omega W^<_n({\bf x},t) = \sum_{\substack{\sigma \in \mathfrak{S}_n \\ \sigma(1)<\sigma(n)}} t^{\des(\sigma)}F_{n, \Des_{\geq 2}(\sigma^{-1})}(\x)\eq
\bq \label{Fgreateq}\omega W^>_n({\bf x},t) = \sum_{\substack{\sigma \in \mathfrak{S}_n \\ \sigma(1)>\sigma(n)}} t^{\des(\sigma)}F_{n, \Asc_{\geq 2}(\sigma^{-1})}(\x) .\eq
\end{thm}

\begin{proof}[Proof of (\ref{Flesseq})]
The first part of the proof is similar to that of \cite[Theorem 3.1]{sw4} and \cite[Theorem 3.1]{e1}. The second part diverges somewhat from these proofs.

{\bf Part 1}:  Given an acyclic orientation $\bar a$ of  the labeled cycle $C_n$, let $E_{\bar a}(C_n)$ be the set of directed edges of $C_n$ under the orientation $\bar a$.  Let $AO^>_n$ be the set of acyclic orientations $\bar{a}$ of $C_n$ such that  
$(n,1) \in E_{\bar{a}}(C_n)$.  For each $\bar{a}\in AO^>_n$, let $W_{\bar{a}} \subseteq W_n$ be 
the set of Smirnov words  $w = w_1 w_2 \cdots w_n$ such that 
\begin{itemize}
\item $w_n < w_1$, 
\item $w_i < w_{i+1} $ if $(i,i+1) \in E_{\bar{a}}(C_n)$ and $i \in [n-1]$, 
\item $w_i > w_{i+1} $ if $(i+1,i) \in E_{\bar{a}}(C_n)$ and $i \in [n-1]$.  
\end{itemize}
Let $\asc(\bar a)$ be the number of edges of 
$E_{\bar{a}}(C_n)$ of the form $(i, i+1)$ for $i \in [n-1]$.  Then by reversing the Smirnov words, we can see that
\begin{equation} \label{eqn1}
W^<_n({\bf x},t) = \sum_{\substack {w \in W_n \\ w_1 > w_n}} x_w t^{\asc(w)} = \sum_{\bar a \in AO^>_n} t^{\asc(\bar a)} \sum_{w \in W_{\bar a}} {\bf x}_w,
\end{equation}
where $\asc(w) := |\{i \in [n-1] : w_i < w_{i+1} \}|$.

Now for each acyclic orientation $\bar a \in AO^>_n$, define a poset $P_{\bar{a}}$ on $[n]$ by letting $i<_{P_{\bar{a}}} j$ if $(i,j) \in E_{\bar a}(C_n)$ and taking the transitive closure of this relation.  Let us define a {\em labeling} of $P_{\bar{a}}$ to be a bijection from $P_{\bar{a}}$ to $[n].$  So a labeling is just a permutation in $\sg_n$.   A labeling $\rho$ is said to be decreasing if $ \rho(i) > \rho(j)$ for all $i <_{P_{\bar{a}}} j$.  For any labeling $\rho$ of $P_{\bar{a}}$, let $L( P_{\bar{a}},\rho)$ be the set of linear extensions of $P_{\bar{a}}$ with the labeling  $\rho$. 

Now fix a decreasing labeling  $\rho_{\bar{a}}$  of $P_{\bar{a}}$ for each $\bar a \in AO^>_n$.  For any subset $S \subseteq [n-1],$ define $n-S = \{i \mid n-i \in S\}.$  Then by the theory of P-partitions \cite[Corollary 7.19.5]{st3}, we have that 
\begin{equation}\label{eqn2}
\sum_{w \in W_{\bar{a}}} {\bf x}_w = \sum_{\sigma \in L(P_{\bar{a}}, \rho_{\bar{a}})} F_{n, n-\Des(\sigma)},
\end{equation}
where $\Des(\sigma)$ is the usual descent set of a permutation, i.e. $\Des(\sigma) = \{ i \in [n-1] : \sigma(i)>\sigma(i+1)\}.$

Let $e:P_{\bar{a}} \rightarrow [n]$ be the identity labeling of $P_{\bar{a}},$ and hence $L(P_{\bar{a}}, e)$ is the set of linear extensions of $P_{\bar{a}}$ with its original labeling.  Note that $\sigma \in L(P_{\bar{a}}, e)$ if and only if $\rho_{\bar{a}}\sigma \in L(P_{\bar{a}}, \rho_{\bar{a}}),$ where $\rho_{\bar{a}}\sigma$ denotes the product of $\rho_{\bar{a}}$ and $\sigma$ in $\mathfrak{S}_n.$  Hence from (\ref{eqn2}), we have 
\begin{equation}\label{eqn3}
\sum_{w \in W_{\bar{a}}} {\bf x}_w = \sum_{\sigma \in L(P_{\bar{a}}, e)} F_{n, n-\Des(\rho_{\bar{a}}\sigma)}.
\end{equation} 

Note that if $\sigma \in L(P_{\bar{a}},e)$ and $\bar a \in AO^>_n$  then $\sigma^{-1}(1) > \sigma^{-1}(n)$.  Conversely, every permutation $\sigma \in \mathfrak{S}_n$ with $\sigma^{-1}(1)>\sigma^{-1}(n)$ is a linear extension in $L(P_{\bar{a}},e)$ for a unique $\bar a \in AO^>_n$. Let $\bar{a}(\sigma)$ denote the unique acyclic orientation of associated with $\sigma.$  
Now combining this with (\ref{eqn1}) and (\ref{eqn3}) yields,
\begin{equation*}
W_n^<({\bf x},t) = \sum_{\substack{\sigma \in \mathfrak{S}_n \\ \sigma^{-1}(1)>\sigma^{-1}(n)}}t^{\asc(\bar{a}(\sigma))} F_{n,n-\Des(\rho_{\bar{a}(\sigma)}\sigma)},
\end{equation*}
where recall $\rho_{\bar{a}(\sigma)}$ is a decreasing labeling of $P_{\bar a(\sigma)}$.
Note that $\asc(\bar a(\sigma)) = \des((\sigma^R)^{-1})$, where $\sigma^R$ is the reverse of $\sigma$.
Hence
\begin{equation}\label{eqn4}
W_n^<(\x,t) = \sum_{\substack{\sigma \in \mathfrak{S}_n \\ \sigma^{-1}(1)>\sigma^{-1}(n)}}
 t^{\des((\sigma^R)^{-1})} F_{n,n-\Des(\rho_{\bar{a}(\sigma)}\sigma)}.
\end{equation}

{\bf Part 2}:  As in the proof of \cite[Theorem 3.1]{sw4}, our next step is to construct a particular decreasing labeling $\tilde\rho_{\bar a(\sigma)}$ of $P_{\bar a(\sigma)}$ for each $\sigma \in L(P_{\bar{a}},e)$.  However since $C_n$ is not the incomparability graph of a poset, the construction used in the proof of \cite[Theorem 3.1]{sw4} does not work in this case.  The construction used here is also quite different from that of \cite[Theorem 3.1]{e1}. Let $p$ be the ``smallest" maximal
element 
 of $P_{\bar a(\sigma)}$
(that is, $p$ is maximal in the  poset $P_{\bar a(\sigma)}$ and is less than all the other maximal elements in the natural order on $ [n]$) and let
$\tilde\rho_{\bar a(\sigma)} (p) = 1$. Now remove $p$ from the poset and let $q$ be the smallest maximal
element of the remaining poset and let $\tilde\rho_{\bar a(\sigma)} (q) = 2$.
Continue this process inductively.  It is clear that $\tilde\rho_{\bar a(\sigma)}$ is a decreasing labeling of $P_{\bar a(\sigma)}$.

{\em Claim.} If $x$ and $y$ are incomparable in $P:=P_{\bar a(\sigma)}$, then $x < y$ implies
$\tilde\rho_{\bar a(\sigma)} (x)< \tilde\rho_{\bar a(\sigma)} (y)$.  

\noindent {\em Proof of Claim.} One can see this by drawing the Hasse diagram of $P$ minus the edge $(n,1)$ as a zig-zag path  on $[n]$ with the elements of $[n]$ increasing as one moves from left to right.  The path consists of up-segments and down-segments.  An up-segment is a maximal chain of $P$ of the form
 $a<_P a+1 <_P \dots <_P a+j$, where $j \ge 1$,  and a down-segment is a maximal chain with top and bottom removed unless it's 1 or n, of the form $a>_P a+1 >_P \dots >_P a+j$, where $j \ge 0$.  Between any two down-segments there is an up-segment. Let $\alpha_i$ be the $i$th segment from the left for each $i$.    One can see that under the labeling 
$\tilde\rho_{\bar a(\sigma)} $, the segment $\alpha_1$ gets the smallest labels,   the segment $\alpha_2$ gets the next smallest labels, and so on.  Now if $x$ and $y$ are incomparable, they are in different segments $\alpha_i$ and $\alpha_j$.  Clearly if $x < y$ then $i <j$, which implies that $x$ gets a smaller label then $y$.  Hence, the claim holds.

Now we show that  
\bq \label {desasceq} \Des(\tilde \rho_{\bar a(\sigma)}\sigma) = [n-1] \setminus \Asc_{\ge 2}(\sigma),\eq 
for all $\sigma \in \sg_n$.  If $i \in \Des(\tilde \rho_{\bar a(\sigma)}\sigma) $ then $\tilde \rho_{\bar a(\sigma)}\sigma(i) > 
\tilde \rho_{\bar a(\sigma)}\sigma(i+1) $. It thus follows from the claim that if $\sigma(i)$ and $\sigma(i+1)$ are incomparable in $P_{\bar a(\sigma)}$ then $\sigma(i )> \sigma(i+1)$, which implies $i \notin \Asc_{\ge 2}(\sigma)$.  On the other hand if $\sigma(i)$ and $\sigma(i+1)$ are comparable in $P_{\bar a(\sigma)}$ then $\sigma(i+1)$ covers $\sigma(i)$ since $\sigma \in  L(P_{\bar a(\sigma)}, e)$.  This implies that either $\sigma(i+1) = \sigma(i)+1$ or $\sigma(i+1) = \sigma(i)-1$.  In either case, $i \notin \Asc_{\ge 2}(\sigma)$.  Thus 
$$\Des(\tilde \rho_{\bar a(\sigma)}\sigma) \subseteq [n-1] \setminus \Asc_{\ge 2}(\sigma).$$

Conversely, if $i \notin   \Des(\tilde \rho_{\bar a(\sigma)}\sigma) $ then $\tilde \rho_{\bar a(\sigma)}\sigma(i) < 
\tilde \rho_{\bar a(\sigma)}\sigma(i+1) $.  It thus follows from the claim that if $\sigma(i)$ and $\sigma(i+1)$ are incomparable in $P_{\bar a(\sigma)}$ then $\sigma(i )< \sigma(i+1)$.  Since $j$ and $j+1$ are comparable in $P_{\bar a(\sigma)}$ for all $j \in [n-1]$, we have $\sigma(i+1) - \sigma(i) \ge 2$. Thus $i \in \Asc_{\ge 2}(\sigma)$. On the other hand if $\sigma(i)$ and $\sigma(i+1)$ are comparable in $P_{\bar a(\sigma)}$ then $\sigma(i)<_{P_{\bar a(\sigma)}}\sigma(i+1)$ since $\sigma \in  L(P_{\bar a(\sigma)}, e)$. But since $\rho$ is a decreasing labeling $\tilde \rho_{\bar a(\sigma)}\sigma(i) > 
\tilde \rho_{\bar a(\sigma)}\sigma(i+1) $, which contradicts our assumption that $i \notin   \Des(\tilde \rho_{\bar a(\sigma)}\sigma) $.  Hence this case is impossible.
We have shown $$\Des(\tilde \rho_{\bar a(\sigma)}\sigma) \supseteq [n-1] \setminus \Asc_{\ge 2}(\sigma),$$
which completes the proof of~(\ref{desasceq}).

Let 
$\omega$ be the involution on the ring of quasisymmetric functions determined by  $\omega F_{n,S} := F_{n,[n-1] \setminus S}$.  Since $\omega$ takes $h_n = F_{n,\emptyset}$ to $e_n= F_{n,[n-1]}$, the involution $\omega$ restricts to the usual involution on the ring of symmetric functions.   
Hence by (\ref{desasceq}), equation~(\ref{eqn4}) becomes
\begin{eqnarray*}
\omega W_n^<({\bf x},t) &=& 
\sum_{\substack{\sigma \in \mathfrak{S}_n \\ \sigma^{-1}(1)>\sigma^{-1}(n)}}t^{\des((\sigma^R)^{-1})} F_{n,n-\Asc_{\ge 2}(\sigma)} 
\\ &=& \sum_{\substack{\sigma \in \mathfrak{S}_n \\ \sigma^{-1}(1)<\sigma^{-1}(n)}} t^{\des(\sigma^{-1})} F_{n,\Des_{\ge 2}(\sigma)} .
\end{eqnarray*}
\end{proof}

\remark{There is an alternative proof of Theorem~\ref{Fth} involving standardization, which will be discussed in a forthcoming paper.}

\begin{proof}[Proof of (\ref{Fgreateq})]  A similar proof can be given here.  One can also use (\ref{Flesseq}) to prove this.  Indeed, by the involution on $W_n$ which reverses Smirnov words, we obtain $$W_n^>(\x,t) = t^{n-1} W_n^<(\x,t^{-1}).$$  By the involution on $\sg_n$, which reverses permutations,
$$ \sum_{\substack{\sigma \in \mathfrak{S}_n \\ \sigma(1)>\sigma(n)}}t^{\des(\sigma)} F_{n,\Asc_{\ge 2}(\sigma^{-1})}  =  \sum_{\substack{\sigma \in \mathfrak{S}_n \\ \sigma(1)<\sigma(n)}} t^{n-1-\des(\sigma)} F_{n, \Des_{\geq 2}(\sigma^{-1})}.$$
The result now follows from (\ref{Flesseq}).
\end{proof}

By combining (\ref{Feq}), (\ref{Flesseq}), and (\ref{Fgreateq}), one gets fundamental quasisymmetric function expansions of the other Smirnov word enumerators $W_n^{=}(\x,t)$, $W_n^{\ne}(\x,t)$, $\tilde W_n(\x,t)$, $\tilde W_n^{\ne}(\x,t)$ and of the chromatic quasisymmetric function $X_{C_n}(\x,t)$.
  The resulting expansion for $\tilde W_n(\x,t)$ has a particularly nice form.
\begin{cor} For all $n \ge 1$,
\bq \label{Fcirceq}\omega \tilde W_n(\x,t) = \sum_{\sigma \in \sg_n} t^{\cdes(\sigma)} F_{n,\Des_{\ge 2}
(\sigma^{-1})}(\x) .\eq
\end{cor}

\begin{proof}We use the fact that $\tilde W_n(\x,t) = t W_n^<(\x,t) + (W_n(\x,t) - W_n^<(\x,t))$.
By (\ref{Feq}) and (\ref{Flesseq}), $$\omega W_n(\x,t) - \omega W_n^<(\x,t) = \sum_{\substack{\sigma \in \sg_n \\ \sigma(1) > \sigma(n) }}  t^{\des(\sigma)} F_{n, \Des_{\geq 2}(\sigma^{-1})}.$$
It follows from this and (\ref{Flesseq}) that
\beq \omega \tilde W_n(\x,t)  &=& \sum_{\substack{\sigma \in \sg_n \\ \sigma(1) < \sigma(n) }}  t^{\des(\sigma)+1} F_{n, \Des_{\geq 2}(\sigma^{-1})} + \sum_{\substack{\sigma \in \sg_n \\ \sigma(1) > \sigma(n) }}  t^{\des(\sigma)} F_{n, \Des_{\geq 2}(\sigma^{-1})} \\
&=& \sum_{\sigma \in \sg_n} t^{\cdes(\sigma)} F_{n, \Des_{\geq 2}(\sigma^{-1})}.\eeq
\end{proof}

There are various ways to specialize expansions in the fundamental quasisymmetric functions to obtain enumerative results.  One way is by  setting $x_i=1$ if $i \in [m]$ and $x_i=0$ otherwise, in a formal power series $f(\x)$.  We denote this specialization by $f(1^m)$. (Another way is discussed in the next section.)   It is not difficult to show that  (see \cite[Section 7.19]{st3}), 
$$F_{n,S}(1^m) = \binom{m+n-1 - |S|}{n},$$ for all $S \subseteq [n-1]$.
It is clear that 
$$W_n(1^m,t) =  \sum_{w \in W_n \cap [m]^n} t^{\des(w)}.$$  Hence by (\ref{Feq})
and the fact that $\omega F_{n,S} = F_{n, [n-1] \setminus S}$,
\bq \label{firstspec1} \sum_{w \in W_n \cap [m]^n} t^{\des(w)} = \sum_{\sigma \in \sg_n } t^{\des(\sigma)} \binom{m + |\Des_{\ge 2}(\sigma^{-1})|}{n},\eq  for all $m,n \in \PP$.
Analogous formulas can be obtained by applying the same specialization to the expansions (\ref{Flesseq}), (\ref{Fgreateq}), and (\ref{Fcirceq}).   The expansions (\ref{Flesseq}) and (\ref{Fcirceq}) yield the following result.
\begin{cor} For all $m,n \ge 1$,
\bq \label{firstspec2} \sum_{\substack{w \in W_n \cap [m]^n \\ w_1 <w_n}} t^{\des(w)} = \sum_{\substack{\sigma \in \sg_n \\ \sigma(1) <\sigma(n)} } t^{\des(\sigma)} \binom{m + |\Des_{\ge 2}(\sigma^{-1})|}{n} \eq 
and 
\bq \label{firstspec3}\sum_{w \in W_n \cap [m]^n} t^{\cdes(w)} = \sum_{\sigma \in \sg_n } t^{\cdes(\sigma)} \binom{m + |\Des_{\ge 2}(\sigma^{-1})|}{n}.\eq 
\end{cor}

\begin{remark} In  \cite{lr1,lr2}, LoBue Tiefenbruck and Remmel study the distribution of  a pair of interesting statistics  on Smirnov words in $[m]^n$ different from $\des$ and $\cdes$.  They use the fact that their statistics are  preserved by a contraction map from unconstrained words to Smirnov words  to transfer their results from Smirnov words to unconstrained words.  Since $\des$ and $\cdes$  are also preserved by the contraction map,  we can also transfer our results to unconstrained words.  \end{remark}

\begin{remark} \label{newproof} We now describe a proof of (\ref{goodtcarleq}) that is different from the proof in \cite{sw1} discussed in the introduction.  Theorem~3.1 of \cite{sw4} gives a fundamental quasisymmetric function expansion of the chromatic quasisymmetric function $X_G(\x,t)$ when $G$  is an incomparability graph. (This reduces to (\ref{Feq}) when $G$ is the path $P_n$.)  In \cite{a}, Athanasiadis proves that  the fundamental quasisymmetric function expansion implies  the conjectured formula (7.14) of \cite{sw4}, which gives a power sum symmetric function expansion of $X_G(\x,t)$ when $G$ is a natural unit interval graph.  It is shown in \cite[Proof of Proposition 7.9]{sw4}  that when $G=P_n$, the power sum symmetric function expansion reduces to (7.15) of \cite{sw4},
which is 
$$\omega X_{P_n}(\x,t) =  \sum_{\lambda \vdash n} \left (A_{l(\lambda)}(t) \prod_{i=1}^{l(\lambda)} [\lambda_i]_t \right )\frac{p_\lambda(\x)}{z_\lambda}.$$  Hence since $W_n(\x,t)= X_{P_n}(\x,t)$, Stembridge's formula (\ref{stemeq})  implies (\ref{goodtcarleq2}), which is equivalent to (\ref{goodtcarleq}).
\end{remark}

\section{Variations of $q$-Eulerian polynomials} \label{qeulersec}

Recall that the Eulerian polynomials $A_n(t)$ have two well-known combinatorial interpretations, which are given by
$$A_n(t):= \sum_{\sigma \in \sg_n} t^{\des(\sigma)} = \sum_{\sigma \in \sg_n} t^{\exc(\sigma)}$$  and that Euler's exponential generating function for the Eulerian polynomials is given by
$$1+ \sum_{n\ge 1} A_n(t) \frac{z^n}{n!} = \frac{(1-t)e^z}{e^{tz} - te^z}.$$

In \cite{sw1} and \cite{sw4}, Shareshian and the second named author obtained combinatorial interpretations of the $q$-Eulerian polynomials $A_n(q,t)$ that satisfy the $q$-exponential generating function formula
\begin{eqnarray} \label{qEulereq}  1+ \sum_{n\ge 1} A_n(q,t) \frac{z^n}{[n]_q!} &=& 
 \frac{(1-t) \exp_q(z)} {\exp_q(tz) - t \exp_q(z)} 
 \\ \nonumber &=& 1+ \frac{(1-t)\sum_{i\ge 2} [i]_t \frac{z^i}{[i]_q!}}
{\exp_q(tz) - t\exp_q(z)}.\end{eqnarray} 
 The interpretation in \cite{sw1} is given by
\bq  \label{interpeq1} A^{\maj,\exc}_n(q,t) =  \sum_{\sigma \in \sg_n} q^{\maj(\sigma)-\exc(\sigma)} t^{\exc(\sigma)} \eq and the interpretation in \cite{sw4} is given by
\bq \label{interpeq2} A_n(q,t) = \sum_{\sigma \in \sg_n} q^{\maj_{\ge 2}(\sigma^{-1})} t^{\des(\sigma)} \eq
where 
 $$\maj(\sigma) = \!\!\!  \sum_{\substack{ i \in [n-1]  \\ \sigma(i+1) > \sigma(i)}} \!\!\!\! i \,\,\,\,\, \mbox{ and } \,\, \,\,\,\maj_{\ge 2}(\sigma) := \!\!\! \sum_{\substack{ i \in [n-1]  \\ \sigma(i+1) - \sigma(i) \ge 2}} \!\!\!\!i.$$ 
  Both $q$-analogs of $A_n(t)$ were obtained by expanding $\omega W_n(\x,t)$ in the fundamental quasisymmetric functions and then taking the stable principal specialization.   A formulation of the expansion obtained in \cite{sw1} yields (\ref{interpeq1}), while the formulation  (\ref{Feq}) obtained in \cite{sw4} yields (\ref{interpeq2}); see \cite[Proof of Theorem 9.7]{sw4} .  From this it follows that the two $q$-analogs are equal.   (A subsequent bijective proof was obtained in \cite{b}.)

  In this section, we use results of the previous sections to obtain  analogs of (\ref{qEulereq})  for  variations of the  interpretation of $A_n(q,t)$ given by (\ref{interpeq2}). 
  The variations are defined by
$$A^<_n(q,t) := \sum_{\substack{\sigma \in \sg_n \\ \sigma(1) < \sigma(n)}} q^{\maj_{\ge 2}(\sigma^{-1})} t^{\des(\sigma)} $$
and
$$\tilde A_n(q,t) := \sum_{\sigma \in \sg_n } q^{\maj_{\ge 2}(\sigma^{-1})} t^{\cdes(\sigma)} .$$
We also obtain nice  formulas for $A^<_n(q,t)$ and $\tilde A_n(q,t)$ evaluated at $n$th roots of unity.

The {\it stable principal specialization} $\sps(G(\x))$ of a quasisymmetric function  $G(\x)$ is obtained from $G(\x)$ by substituting $q^{i-1}$ for $x_i$ for all $i \ge 1$. By \cite[Lemma 5.2]{gr}, 
$$\sps(F_{n,S}(x)) = \frac{\sum_{i \in S} q^i} {(1-q)(1-q^2) \cdots (1-q^n)} $$
for all  $S\subseteq [n-1]$.
Hence by (\ref{Feq}), (\ref{Flesseq}), and (\ref{Fcirceq}), respectively,
\bq \label{psWeq1} \sps(\omega W_n(\x,t))  = \frac{A_n(q,t)}{(1-q)(1-q^2) \cdots (1-q^n)} \eq
\bq \label{psWeq2} \sps(\omega W^<_n(\x,t))  = \frac{A^<_n(q,t)}{(1-q)(1-q^2) \cdots (1-q^n)} \eq
\bq \label{psWeq3} \sps(\omega \tilde W_n(\x,t))  = \frac{\tilde A_n(q,t)}{(1-q)(1-q^2) \cdots (1-q^n)} .\eq

In \cite{sw1,sw4}, first $\omega$ is applied to both sides of (\ref{goodtcarleq2}), then the stable principal specialization is taken using (\ref{psWeq1}), and  finally  $z$ is replaced by $(1-q)z$ resulting in 
(\ref{qEulereq}).  By doing the same to (\ref{lessWeeq}) and (\ref{circWeeq2}), using (\ref{psWeq2}) and (\ref{psWeq3}), respectively, we obtain the following result.

\begin{thm} We have
\bq \label {explesseq} \sum_{n\ge 1} A^<_n(q,t) \frac{z^n}{[n]_q!} = \frac{(1-t) \frac{\partial}{\partial t}\sum_{i\ge 2}  [i]_t\frac{z^i}{[i]_q!}}{\exp_q(tz) - t\exp_q(z)}
\eq

\bq \label {expcirceq} \sum_{n\ge 1} \tilde A_n(q,t) \frac{z^n}{[n]_q!}= \frac{(1-t)\frac{\partial }{\partial t}\exp_q(tz)}
{\exp_q(tz) - t\exp_q(z)}.
\eq
\end{thm}

In \cite[Corollary 6.2]{ssw},  Sagan, Shareshian and the second named author show that for every $n$th root of unity  $\xi$, the coefficients of  the polynomial
$A_n(\xi,t)$ are positive integers.  More precisely, they show that 
if $k | n$ and $\xi_k$ is any primitive $k$th root of unity   then
\bq \label{unityeq} A_n(\xi_k,t) = A_{\frac{n} {k}}(t) \,\, [k]^{\frac{n} {k}}_t .\eq 
Consequently, $A_n(\xi_k,t)$ is a  palindromic, unimodal polynomial in $\N[t]$.  Here we prove analogous results for  other Smirnov word enumerators.

A key tool in the proof of (\ref{unityeq})  is the following result, which is  implicit in \cite{d} and  stated explicitly in \cite{ssw}.

\begin{lemma}[see {\cite[Proposition 3.1]{ssw}}]  \label{thdes} Let $R$ be a commutative ring. Suppose $u(q) \in R[q]$ and there exists  a homogeneous symmetric function $U(\x)$ of degree $n$ with coefficients in $R$ such that   
 $$u(q) =(1-q)(1-q^2)\dots(1-q^n)\,\, \sps (U(\x)).$$ If $k|n$  then $u(\xi_k)$ is the coefficient of $z_{(k^{\frac n k})}^{-1} p_{(k^{\frac n k})}$ in the expansion of $U(\x)$ in the power sum basis.  
\end{lemma}

In \cite{ssw}, (\ref{unityeq}) is proved by setting $R=\Q[t]$ and $U(\x) = \omega W_n(\x,t)$ in  Lemma~\ref{thdes}.  By (\ref{psWeq1}), $u(q)=A_n(q,t)$.  Hence it follows from  Lemma~\ref{thdes} that $A_n(\xi_k,t)$ equals the coefficient of $z_{(k^{\frac n k})}^{-1} p_{(k^{\frac n k})}$ in the expansion of $\omega W_n(\x,t) $ in the power sum basis, which by (\ref{powereq}) equals $A_{\frac{n}{k} }(t)\, [k]_t^{\frac{n}{k}}$.
We use a similar argument to obtain the following result.  Indeed, to prove (\ref{unity1eq}) below, we set $U(\x) = \omega W^{<}_n(\x,t)$ and use (\ref{psWeq2}) and Theorem~\ref{powlesth}.   To prove (\ref{unity2eq}) below, we set $U(\x) =\omega \tilde W_n(\x,t)$ and use (\ref{psWeq3}) and Theorem~\ref{pcircth}.

\begin{thm} Let $n \ge 2$ and $k | n$. If $\xi_k$ is any primitive $k$th root of unity   then

\bq \label{unity1eq} A^{<}_n(\xi_k,t) = \frac{d }{dt} ( tA_{\frac{n}{k} -1}(t)\, [k]_t^{\frac{n}{k}}) 
\eq 
and
\bq \label{unity2eq} \tilde A_n(\xi_k,t) = n t^{k} A_{\frac{n}{k} -1}(t)\, [k]_t^{\frac{n}{k}-1} .
\eq
Consequently, $A^{<}_n(\xi_k,t), \tilde A_n(\xi_k,t) \in \N[t]$ and  $\tilde A_n(\xi_k,t)$ is  palindromic and unimodal.
\end{thm}

\begin{cor}For all $n \ge 2$,
\bq \label{unity4eq} A^{<}_n(1,t) = \frac{d }{dt} ( tA_{n -1}(t))
\eq
and
\bq \label{unity3eq} \tilde A_n(1,t) = n t A_{n-1}(t) .
\eq
\end{cor}

Equations~(\ref{unity4eq}) and~({\ref{unity3eq}) have  elementary bijective proofs.  Indeed, 
for each $\sigma \in \sg_{n}$, such that $\sigma(n) = n$, let
$\mathcal C_\sigma$ be the set of circular rearrangements of $\sigma$. Clearly, $|\mathcal C_\sigma|= n$ and for each $\tau \in  \mathcal C_\sigma$, we have $\cdes(\tau) = \des(\sigma)+1$. 
Hence, 
\begin{eqnarray*} \tilde A_n(1,t) &=& \sum_{\substack {\sigma \in \sg_n \\ \sigma(n)= n}} \sum_{\tau \in \mathcal C_\sigma} t^{\cdes(\tau)}
\\ &=&  \sum_{\substack {\sigma \in \sg_n \\ \sigma(n)= n}} n t^{\des(\sigma)+1}
\\&=& n\sum_{\sigma \in \sg_{n-1}}  t^{\des(\sigma)+1}
\\ &=& ntA_{n-1}(t).
\end{eqnarray*}

Now for each $\sigma \in \sg_{n}$, such that $\sigma(n) = n$, let
$$\mathcal C^{<}_\sigma:=\{\tau \in \mathcal C_\sigma: \tau(1) < \tau(n)\} .$$ 
Clearly, $|\mathcal C^{<}_\sigma| = \des(\sigma) +1$
and for each $\tau \in  \mathcal C^{<}_\sigma$, we have $\des(\tau) = \des(\sigma)$.   Hence 
\begin{eqnarray*} A^{<}_n(1,t) &=& \sum_{\substack {\sigma \in \sg_n \\ \sigma(n)= n}} \sum_{\tau \in \mathcal C^{<}_\sigma} t^{\des(\tau)}
\\ &=&  \sum_{\substack {\sigma \in \sg_n \\ \sigma(n)= n}} (\des(\sigma)+1) t^{\des(\sigma)}
\\&=& \sum_{\sigma \in \sg_{n-1}} (\des(\sigma)+1) t^{\des(\sigma)}
\\ &=& \frac{d}{dt}(tA_{n-1}(t)).
\end{eqnarray*}

By combining (\ref{unityeq}) with (\ref{unity1eq}) and with (\ref{unity2eq}), we obtain the following generalization of the previous corollary.

\begin{cor} \label{unitycor} Let $n \ge 2$ and $k | n$. If $\xi_k$ is any primitive $k$th root of unity   then
$$A_n^{<}(\xi_k,t) = \frac{d}{dt}(t[k]_t A_{n-k}(\xi_k,t))$$
and
\bq \tilde A_n(\xi_k,t) = n t^k A_{n-k}(\xi_k,t) .\eq
\end{cor}

\end{document}